\documentclass[10 pt,draft]{article}
\title{Stability in the homology of congruence subgroups}
\author{Andrew Putman\footnote{Supported in part by NSF grant DMS-1005318}\vspace{-6pt}}

\usepackage{amsmath}
\usepackage{amssymb}
\usepackage{amsthm}
\usepackage{amscd}
\usepackage{epsfig}
\usepackage[hmargin=1.25in, vmargin=1.0in]{geometry}
\usepackage{amsfonts}
\usepackage[font=small,format=plain,labelfont=bf,up,textfont=it,up]{caption}
\usepackage{amscd}
\usepackage{stmaryrd}
\usepackage{type1cm}
\usepackage{calc}
\usepackage{mathptmx}  
\usepackage[all,cmtip]{xy}
\usepackage{ytableau}

\theoremstyle{plain}
\newtheorem{theorem}{Theorem}[section]
\newtheorem{maintheorem}{Theorem}
\newtheorem{proposition}[theorem]{Proposition}
\newtheorem{lemma}[theorem]{Lemma}

\newtheorem{corollary}[theorem]{Corollary}

\newtheorem{claims}{Claim}
\newtheorem{step}{Step}

\newcommand\BeginClaims{\setcounter{claims}{0}}

\newcommand\BeginSteps{\setcounter{step}{0}}

\theoremstyle{definition}
\newtheorem{assumption}{Assumption}
\newtheorem*{definition}{Definition}
\newtheorem*{notation}{Notation}

\theoremstyle{remark}
\newtheorem*{remark}{Remark}

\newtheorem*{example}{Example}

\DeclareMathOperator{\Hom}{Hom}

\DeclareMathOperator{\Ker}{ker}
\DeclareMathOperator{\Coker}{coker}
\DeclareMathOperator{\Image}{Im}

\DeclareMathOperator{\Sp}{Sp}
\DeclareMathOperator{\GL}{GL}
\DeclareMathOperator{\SL}{SL}

\newcommand\SLLie{\ensuremath{\mathfrak{sl}}}

\newcommand\R{\ensuremath{\mathbb{R}}}

\newcommand\Z{\ensuremath{\mathbb{Z}}}
\newcommand\Q{\ensuremath{\mathbb{Q}}}

\newcommand\Field{\ensuremath{\mathbb{F}}}
\DeclareMathOperator{\Char}{char}

\DeclareMathOperator{\HH}{H}
\newcommand\RH{\ensuremath{\widetilde{\HH}}}
\newcommand\Chain{\ensuremath{\text{C}}}

\newcommand\Span[1]{\ensuremath{\langle #1 \rangle}}

\DeclareMathOperator{\Dim}{dim}
\newcommand\Set[2]{\ensuremath{\{\text{#1 $|$ #2}\}}}

\DeclareMathOperator{\stab}{st}
\newcommand\hatstab{\ensuremath{\widehat{\text{st}}}}
\DeclareMathOperator{\Ind}{Ind}
\DeclareMathOperator{\Res}{Res}
\DeclareMathOperator{\SR}{SR}
\DeclareMathOperator{\ColStab}{ColStab}
\newcommand\Bound[1]{\ensuremath{#1}}

\newcommand\Split{\ensuremath{\mathcal{SB}}}
\newcommand\Poset{\ensuremath{\mathcal{P}}}

\newcommand{\Filter}[2]{\ensuremath{\mathcal{F}_{#1} #2}}
\newcommand{\Alternate}[1]{\ensuremath{\mathcal{A}_{#1}}}
\newcommand{\Trivial}[1]{\ensuremath{\mathcal{T}_{#1}}}
\newcommand{\Stab}[2]{\ensuremath{\mathcal{C}(#1 \rightarrow #2)}}
\newcommand{\Stabb}[3]{\ensuremath{\mathcal{C}(#1 \stackrel{#3}{\rightarrow} #2)}}
\newcommand{\Induce}[2]{\ensuremath{\text{IA}_{#2}(#1)}}
\newcommand{\InduceTriv}[2]{\ensuremath{\text{I}_{#2}(#1)}}
\newcommand{\twoheadlongrightarrow}{\relbar\joinrel\twoheadrightarrow}

\begin{document}

\maketitle

\begin{abstract}
The homology groups of many natural sequences of groups $\{G_n\}_{n=1}^{\infty}$
(e.g.\ general linear groups, mapping class groups, etc.)
stabilize as $n \rightarrow \infty$.  Indeed, there is a well-known machine
for proving such results that goes back to early work of Quillen.  Church
and Farb discovered that many sequences of groups whose homology groups
do not stabilize in the classical sense actually stabilize in some sense
as representations.  They called this phenomena {\em representation stability}.
We prove that the homology groups of congruence subgroups of $\GL_n(R)$ (for
almost any reasonable ring $R$) satisfy a strong version of representation stability that we call
{\em central stability}.  The definition of central stability is very different
from Church-Farb's definition of representation stability (it is defined via
a universal property), but we prove that it implies representation stability.  
Our main tool is a new machine for proving central stability that is analogous
to the classical homological stability machine.
\end{abstract}

\section{Introduction}
\label{section:introduction}

\paragraph{Arithmetic groups and Borel stability.}
The homology groups of arithmetic groups like $\SL_n(\Z)$ play important
roles in algebraic K-theory, the theory of locally symmetric spaces, and
the study of automorphic forms.  The fundamental theorem about them
is the Borel stability theorem \cite{BorelStability}, which among
other things calculates $\HH_k(\SL_n(\Z);\Q)$ for $n \gg k$.  The answer
turns out to be independent of $n$ for $n \gg k$, so one says that 
$\HH_k(\SL_n(\Z);\Q)$ {\em stabilizes}.  This stability property was
later generalized to $\Z$-coefficients by Maazen \cite{MaazenThesis}.

For $\ell \geq 2$, the {\em level $\ell$ congruence subgroup}
of $\SL_n(\Z)$, denoted $\SL_n(\Z,\ell)$, is the kernel of the 
natural map $\SL_n(\Z) \rightarrow \SL_n(\Z/\ell)$.  Both 
$\SL_n(\Z)$ and $\SL_n(\Z,\ell)$ are lattices in $\SL_n(\R)$, and Borel's
theorem applies to all such lattices.  Amazingly, the output of Borel's 
theorem does not depend on the lattice one is investigating,
so for $n \gg k$ we have 
$\HH_k(\SL_n(\Z,\ell);\Q) \cong \HH_k(\SL_n(\Z);\Q)$.  In particular,
$\HH_k(\SL_n(\Z,\ell);\Q)$ stabilizes.  

\paragraph{Torsion.}
However, the torsion in the homology of $\SL_n(\Z,\ell)$ is far more complicated.  A theorem
of Charney \cite{CharneyCongruence} says that if $\Field$ is a field such that $\Char(\Field)$
does not divide $\ell$, then $\HH_k(\SL_n(\Z,\ell);\Field)$ stabilizes.  But this restriction
on $\Char(\Field)$ is necessary -- Lee and Szczarba \cite{LeeSzczarba} proved
that $\HH_1(\SL_n(\Z,\ell);\Z) \cong (\Z/\ell)^{n^2-1}$ for $n \geq 3$, which gets larger 
and larger as $n$ increases.  There are few concrete calculations of any of the higher
integral homology groups of $\SL_n(\Z,\ell)$, and their structure remains largely
a mystery.  Indeed, aside from certain small values of $n$ and $\ell$, even
$\HH_2(\SL_n(\Z,\ell);\Z)$ is not known.  Moreover, the sporadic calculations that do
exist display no obvious patterns.

In this paper, we give a precise description of how $\HH_k(\SL_n(\Z,\ell);\Field)$ changes as $n$ increases 
for fields $\Field$ whose characteristic is positive but not too small.  
An easily stated consequence of our results is the following.

\begin{maintheorem}
\label{maintheorem:slnzlpoly}
For $k \geq 1$, there exists some $P_k$ such that if $\Field$ is a field with
$\Char(\Field) \geq P_k$, then for all $\ell \geq 2$, there exists a polynomial $\phi(n)$ such that
$\phi(n) = \Dim_{\Field} (\HH_k(\SL_n(\Z,\ell);\Field))$
for $n \gg 0$.
\end{maintheorem}

\noindent
This generalizes Lee and Szczarba's theorem, which says that we can take $\phi(n)=n^2-1$ if $k=1$ and
$\Char(\Field)$ divides $\ell$.  Bounds for the constant $P_k$ can be easily extracted from
our results; see below.

\begin{remark}
Our restriction on $\Char(\Field)$ depends only on $k$, not on $\ell$, so Theorem \ref{maintheorem:slnzlpoly}
can be applied in situations where $\Char(\Field)$ divides $\ell$.
We conjecture that the restrictions on $\Char(\Field)$ in Theorem \ref{maintheorem:slnzlpoly} and
in Theorem \ref{maintheorem:congruence} below are unnecessary.
\end{remark}

\paragraph{More general rings.}
In fact, our techniques give results about congruence subgroups of $\GL_n(R)$
for very general rings $R$.  It suffices for $R$ to
be a commutative Noetherian ring of finite Krull dimension.  This includes rings of integers in
algebraic number fields, but it also includes more exotic rings like $\Z[x_1,\ldots,x_m]$ and
$\Field_p[t]$.  For a ring $R$ of this type, van der Kallen \cite{vanDerKallen} proved that
$\HH_k(\GL_n(R);\Z)$ stabilizes as $n$ increases.  Moreover, 
Charney \cite{CharneyCongruence} generalized Borel's theorem to prove that 
the $\Q$-homology groups of {\em finite-index} congruence subgroups of $\GL_n(R)$ stabilize.
However, there also exist infinite-index congruence subgroups.  Their homology groups can
display interesting phenomena even over $\Q$, and our results cover these cases as well.  There
are essentially no known computations of the homology groups of congruence subgroups of $\GL_n(R)$
for these more general rings.

\paragraph{Group actions.}
The key to understanding the homology groups of congruence subgroups is to observe
that they are not just naked abelian groups, but also representations.
Indeed, $\HH_k(\SL_n(\Z,\ell);\Z)$ is acted upon by
$$\SL_n(\Z/\ell) = \SL_n(\Z) / \SL_n(\Z,\ell).$$
This a general phenomena : if $G$ is a normal subgroup of 
$\Gamma$, then the conjugation action of $\Gamma$ on $G$ induces
an action of $\Gamma / G$ on $\HH_k(G)$.  Here we are using the fact that
the conjugation action of $G$ on itself induces the trivial action on $\HH_k(G)$.  
Lee and Szczarba's theorem actually identifies the $\SL_n(\Z/\ell)$-action
on $\HH_1(\SL_n(\Z,\ell);\Z)$.  They prove that there is an $\SL_n(\Z,\ell)$-equivariant
isomorphism $\HH_1(\SL_n(\Z,\ell);\Z) \cong \SLLie_n(\Z/\ell)$, where $\SLLie_n(\Z/\ell)$
is the abelian group of $n \times n$ matrices over $\Z/\ell$ with trace $0$ and
$\SL_n(\Z/\ell)$ acts on $\SLLie_n(\Z/\ell)$ by conjugation.

\paragraph{Representation stability.}
In summary, $\HH_1(\SL_n(\Z,\ell);\Z)$ does not stabilize as an abelian group,
but in some sense it stabilizes as a representation.  In
\cite{ChurchFarbStability}, Church and Farb introduced the notion of
{\em representation stability} to make this observation and others like it
precise.  The basic idea is to give a ``stabilization recipe'' for producing the next
representation in a sequence from the previous ones.  Church and Farb
proposed stabilization recipes for many different kinds of representations
and proved that many natural sequences of representations obeyed their rules.

\paragraph{New machine.}
There is a well-known machine of Quillen which has been used by many people to prove ordinary homological stability
for different kinds of groups (see, e.g.,\ \cite{HatcherWahl}).
We construct a version of this machine for representation stability.  This involves a delicate interplay between 
equivariant homology and representation theory (in both characteristic $0$ and $p$).  The
key difficulty is that we need a much stronger inductive hypothesis to make the
proof work than is provided by representation stability (which, in particular, is 
not functorial in any natural sense).  

The solution to this is the new notion of {\em central stability}.
This is defined in terms of a representation-theoretic universal property.
In many ways, its definition is simpler than Church and Farb's definition of 
representation stability, which is defined in terms of the structure
of the irreducible representations of the groups in question.  Nonetheless,
it gives much tighter control over the representation theory than does representation stability.
In Theorem \ref{maintheorem:centraltospecht} below, we will prove that for finite-dimensional
representations over a field of characteristic $0$, central stability implies representation stability.

\begin{remark}
It is not obvious from its definition that central stability provides as much
control over the representations in question as it does.  Indeed, nearly
half of this paper is devoted to proving a certain representation-theoretic 
``regularity'' theorem concerning central stability; see 
Proposition \ref{proposition:exactness}.
\end{remark}

Our machine can be applied in other contexts too.  For instance, 
in \cite{PutmanModStability}, the author proves that the homology groups of mapping class groups of
arbitrary connected manifolds $M^n$ ($n \geq 2$) 
with marked points and nonempty boundary satisfy central stability.  We
emphasize that these manifolds are completely general -- they are not assumed to
be compact or even of finite type.  

\begin{remark}
For representations over $\Sp_{2g}(\Z)$, some of the ideas in the theory of representation stability are also contained in unpublished work of
Hain on the cohomology of the Torelli group from the 1990's.
\end{remark}

\begin{remark}
After a draft of this paper was circulated, we learned that Church, Ellenberg, and Farb
\cite{ChurchEllenbergFarb} have developed a theory of what they call {\em FI-modules},
which at least in characteristic $0$ seem to be closely related to the notion of central stability.
\end{remark}

\paragraph{Representation stability \`{a} la Church--Farb.}
Before stating our theorems, we need to give a precise definition of central
stability.  The definition of representation stability 
introduced by Church and Farb suffers from three defects.
\begin{itemize}
\item It is only appropriate for finite-dimensional representations over
a field of characteristic $0$.
\item It is a bit ad-hoc, and requires a ``consistent naming scheme''
for the irreducible representations.
\item It does not pin down the maps between the representations in a sequence.
\end{itemize}
Central stability overcomes these difficulties; in particular, its definition
makes no reference to the characteristic of the field or the dimensions of
the representation.

\paragraph{Central stability, motivation.}
Let us return to the example of $\SL_n(\Z,\ell)$.  Fixing a field $\Field$ and some $k \geq 1$, set
$V_n = \HH_k(\SL_n(\Z,\ell);\Field)$.  We will view $V_n$ as a representation of the symmetric group
$S_n$, which acts on $V_n$ via the conjugation action of permutation matrices on $\SL_n(Z,\ell)$.

\begin{remark}
Of course, one would ideally want a description of $\HH_k(\SL_n(\Z,\ell);\Field)$ as
a representation of the group $\SL_n(\Z/\ell)$; however, especially in finite
characteristic the representation theory of $\SL_n(\Z/\ell)$ is extremely complicated
and difficult to work with.
\end{remark}

How should we expect the $S_{n+1}$-representation $V_{n+1}$
to be constructed from the $S_n$-representation $V_n$?  A first guess is that $V_{n+1}$ is
the induced representation $\Ind_{S_n}^{S_{n+1}} V_n$.  Unfortunately, this cannot be the case.  Let
$P \in \GL_{n+1}(\Z)$ be the permutation matrix corresponding to the transposition $(n,n+1) \in S_{n+1}$.
We then have $P \phi P^{-1} = \phi$ for all $\phi \in \SL_{n-1}(\Z,\ell) \subset \SL_{n+1}(\Z,\ell)$.  This
implies that $P$ must act trivially on the image of $V_{n-1}$ in $V_{n+1}$.  In general,
this will not be the case in the induced representation.

\paragraph{Central stabilization.}
It turns out that in a stable range, this is all that goes wrong.  To formalize this,
we now introduce our ``stabilization recipe''.  Let $\phi_{n-1} : V_{n-1} \rightarrow V_n$ 
be an $S_{n-1}$-equivariant map from a representation of $S_{n-1}$ to a representation
of $S_n$.  The {\em central stabilization} of $\phi_{n-1}$, 
denoted $\Stabb{V_{n-1}}{V_n}{\phi_{n-1}}$, is the $S_{n+1}$-representation
which is the largest quotient of $\Ind_{S_n}^{S_{n+1}} V_n$ such that $(n,n+1)$ acts
trivially on the image of $V_{n-1}$.  More precisely, let
$W = \Ind_{S_n}^{S_{n+1}} V_n$.  Composing $\phi_{n-1}$ with the natural inclusion $V_n \hookrightarrow W$,
we obtain an $S_{n-1}$-equivariant map $\phi'_{n-1} : V_{n-1} \rightarrow W$.  Then
$\Stabb{V_{n-1}}{V_n}{\phi_{n-1}} = W / U$, where $U$ is the span
of the $S_{n+1}$-orbit of the set
$$\Set{$\vec{v} - (n,n+1) \cdot \vec{v}$}{$\vec{v} = \phi'_{n-1}(\vec{v}')$ for some $\vec{v}' \in V_{n-1}$}.$$
Observe that there is a natural $S_n$-equivariant map $V_n \rightarrow \Stabb{V_{n-1}}{V_n}{\phi_{n-1}}$.

\paragraph{Examples.}
Here are some examples to convince the reader that this is a natural concept.

\begin{example}
For $n \geq 1$, let $\Trivial{n} \cong \Field$ be the trivial $S_n$-representation.  Then
$\Stab{\Trivial{n-1}}{\Trivial{n}} = \Trivial{n+1}$.  Indeed, $W = \Ind_{S_n}^{S_{n+1}} \Trivial{n}$ is the
{\em permutation representation} $\mathcal{P}_{n+1}$, i.e.\
the vector space consisting of $\Field$-linear combinations of formal symbols 
$\Set{$[i]$}{$1 \leq i \leq n+1$}$.  The group
$S_{n+1}$ acts on $\mathcal{P}_{n+1}$ in the obvious way.  The image of $\Trivial{n-1}$ in $W$ is
the span of $[n+1]$.  Defining $U \subset W$ to be the span of
$$S_{n+1} \cdot \{[n+1]-[n]\} = \Set{$[i]-[j]$}{$1 \leq i,j \leq n+1$ distinct},$$
we have $\Stab{\Trivial{n-1}}{\Trivial{n}} = W/U = \Trivial{n+1}$.
\end{example}

\begin{example}
For $n \geq 2$, we have $\Stab{\mathcal{P}_{n-1}}{\mathcal{P}_n} = \mathcal{P}_{n+1}$.
Indeed, $W = \Ind_{S_n}^{S_{n+1}} \mathcal{P}_n$ is the vector space consisting
of $\Field$-linear combinations of formal symbols
$\Set{$[i,j]$}{$1 \leq i,j \leq n+1$, $i \neq j$}$
and the obvious $S_{n+1}$-action.  The image of $\mathcal{P}_{n-1}$ in $W$ is spanned by
$\Set{$[i,n+1]$}{$1 \leq i \leq n-1$}$.  Defining $U \subset W$ to be the span of
$$S_{n+1} \cdot \Set{$[i,n+1] - [i,n]$}{$1 \leq i \leq n-1$} = \Set{$[i,j]-[i,k]$}{$1 \leq i,j,k \leq n+1$ distinct},$$
we have $\Stab{\mathcal{P}_{n-1}}{\mathcal{P}_n} = W/U = \mathcal{P}_{n+1}$.
\end{example}

\paragraph{Central stability, definition.}
We finally define central stability.
For each $n$, let $V_n$ be a representation of $S_n$ over $\Field$ and
let $\phi_n : V_n \rightarrow V_{n+1}$ be a linear map which is $S_n$-equivariant.
We will call the sequence
$$V_1 \stackrel{\phi_1}{\longrightarrow} V_2 \stackrel{\phi_2}{\longrightarrow} V_3 \stackrel{\phi_3}{\longrightarrow} V_4 \stackrel{\phi_4}{\longrightarrow} \cdots$$
a {\em coherent sequence} of representations of the symmetric group.  We will say that
our coherent sequence is {\em centrally stable} starting at $N \geq 2$ if
for all $n \geq N$, we have $V_{n+1} = \Stabb{V_{n-1}}{V_n}{\phi_{n-1}}$ and $\phi_n$ is the natural map
$V_n \rightarrow \Stabb{V_{n-1}}{V_n}{\phi_{n-1}}$.

\paragraph{Main theorem.}
We now turn to general congruence subgroups.  Let $R$ be a ring with a unit (not necessarily commutative)
and let $q$ be a $2$-sided ideal of $R$.  The {\em level $q$ congruence subgroup} of $\GL_n(R)$,
denoted $\GL_n(R,q)$, is the kernel of the map $\GL_n(R) \rightarrow \GL_n(R/q)$.  The maps
$\GL_n(R) \rightarrow \GL_n(R/q)$ need not be surjective, so there might not exist a $\GL_n(R/q)$-action
on the homology groups of $\GL_n(R,q)$.  However, $S_n$ is embedded in $\GL_n(R)$ as the group
of permutation matrices.  Restricting the conjugation action of $\GL_n(R)$ on $\GL_n(R,q)$ to $S_n$,
we get an action of $S_n$ on $\GL_n(R,q)$ and thus on $\HH_{\ast}(\GL_n(R,q);\Field)$.

We will need to assume that the pair $(R,q)$ satisfies the Bass's {\em stable range condition} $\SR_{d+2}$ 
for some $d \geq 0$ (see \S \ref{section:congruence}).  This condition depends on $q$; however,
Bass also defined a stable range condition $\SR_{d+2}$ for rings $R$ and proved that if $R$ satisfies
$\SR_{d+2}$, then $(R,q)$ satisfies $\SR_{d+2}$ for all $2$-sided ideals $q$.  Almost any reasonable
ring $R$ satisfies $\SR_{d+2}$ for some $d$.  For example, fields satisfy $\SR_2$ and $\Z$ satisfies
$\SR_3$.  More generally, in \cite[Theorem V.3.5]{BassKTheory} Bass proved
that if $R$ is a commutative Noetherian ring with Krull dimension $d$,
then $R$ satisfies $\SR_{d+2}$.

Van der Kallen \cite{vanDerKallen} proved that if $R$ satisfies $\SR_{d+2}$ for
some $d$, then the homology groups of $\GL_n(R)$ are stable in the classical sense.

\begin{maintheorem}
\label{maintheorem:congruence}
Let $R$ be a ring with unit and let $q$ be a $2$-sided ideal of $R$.  Assume that $(R,q)$ satisfies
$\SR_{d+2}$.  Fix $k \geq 1$, and assume that either $\Char(\Field)=0$ or $\Char(\Field) \geq \Bound{(d+8) 2^{k-1}-3}$.  
Then the sequence
$$\HH_k(\GL_1(R,q);\Field) \rightarrow \HH_k(\GL_2(R,q);\Field) \rightarrow \HH_k(\GL_3(R,q);\Field) \rightarrow \cdots$$
of representations of the symmetric group is centrally 
stable with stability beginning at $\Bound{(d+8) 2^{k-1} - 4}$.
\end{maintheorem}

\begin{remark}
We want to emphasize that in Theorem \ref{maintheorem:congruence} we are {\em not} assuming that
the homology groups of $\GL_n(R,q)$ are finite-dimensional.
\end{remark}

\begin{remark}
In \cite[\S 5.4]{CharneyCongruence}, Charney
gives a number of conditions on $\Field$ and $(R,q)$ which ensure that the groups
$\HH_k(\GL_n(R,q);\Field)$ are stable in the classical sense.  For instance, she proves
that this is true if $R$ satisfies $\SR_{d+2}$ and $R/q$ is finite and $\Char(\Field)=0$.
However, we should emphasize that we are not assuming that $R/q$ is finite, so our
congruence subgroups need not be finite-index and Theorem \ref{maintheorem:congruence} has
content even if $\Char(\Field)=0$.
\end{remark}

If $R$ is a commutative ring, then there is a determinant map $\GL_n(R) \rightarrow R^{\ast}$ and
we can define $\SL_n(R)$ and $\SL_n(R,q)$ in the obvious way.  
In this notation, the congruence subgroup
$\SL_n(\Z,\ell)$ of $\SL_n(\Z)$ is $\SL_n(\Z,\ell \Z)$.  We then have the following.

\begin{maintheorem}
\label{maintheorem:congruencesl}
Let $R$ be a commutative ring with unit and let $q$ be an ideal of $R$.  Assume that $(R,q)$ satisfies
$\SR_{d+2}$.  Fix $k \geq 1$, and assume that either $\Char(\Field)=0$ or $\Char(\Field) \geq \Bound{(d+8) 2^{k-1}-3}$.
Then the sequence
$$\HH_k(\SL_1(R,q);\Field) \rightarrow \HH_k(\SL_2(R,q);\Field) \rightarrow \HH_k(\SL_3(R,q);\Field) \rightarrow \cdots$$
of representations of the symmetric group is centrally
stable with stability beginning at $\Bound{(d+8) 2^{k-1} - 4}$.
\end{maintheorem}

\noindent
There is a huge literature on the finiteness properties of groups like $\GL_n(R,q)$ for special choices of
$R$ and $q$.  See, for instance, \cite{BuxGramlichWitzel}.  However, we are not aware of any
concrete calculations of even their first homology groups aside from Lee and Szczarba's calculation
of $\HH_1(\SL_n(\Z,\ell);\Z)$.

\paragraph{Central stability implies polynomial dimensions.}
A fundamental insight of Church, Ellenberg, and Farb \cite{ChurchEllenbergFarb} is that there
is a close relationship between a coherent sequence of representations being representation stable 
and the dimension of the $n^{\text{th}}$ term in the sequence being given by a polynomial in
$n$ for $n \gg 0$.  This inspires the following theorem.

\begin{maintheorem}
\label{maintheorem:polynomial}
Let
$$V_1 \longrightarrow V_2 \longrightarrow V_3 \longrightarrow V_4 \longrightarrow \cdots$$
be a coherent sequence of representations of the symmetric group over a field $\Field$
which is centrally stable with stability starting at $N$.  Assume that either $\Char(\Field)=0$
or $\Char(\Field) \geq 2N+2$.  Then one of the following holds.
\begin{itemize}
\item $\Dim_{\Field} V_n = \infty$ for $n \geq 2N+1$.
\item There exists a polynomial $\phi(n)$ such that $\phi(n) = \Dim_{\Field} V_n$ for $n \geq 2N+1$.
\end{itemize}
\end{maintheorem}

\noindent
Theorem \ref{maintheorem:polynomial} can be combined with Theorem 
\ref{maintheorem:congruence}--\ref{maintheorem:congruencesl}
to deduce theorems analogous to Theorem \ref{maintheorem:slnzlpoly} for the homology groups
of $\GL_n(R,q)$ and $\SL_n(R,q)$.  These results must allow for the
possibility that the dimensions of the relevant homology groups are infinite (as in the
first possibility in Theorem \ref{maintheorem:polynomial}).
Theorem \ref{maintheorem:slnzlpoly} does not allow for this possibility; in fact,
Borel and Serre \cite{BorelSerre} proved that $\HH_k(\SL_n(\Z,\ell);\Field)$ is always
finite-dimensional.

\paragraph{Specht stability.}
To relate central stability to the fine structure of the representation theory
of $S_n$ (and thus to Church and Farb's notion of representation stability), we will
give in \S \ref{section:spechtstability} below a definition of what we call {\em Specht stability}.
For finite-dimensional representations over a field of characteristic $0$, Specht
stability is a strengthening of Church and Farb's notion of representation stability.  This
definition is related to (and implies) Church's notation of {\em monotonicity} for stable
representations, which he defined in \cite{ChurchConfiguration}.  We will prove the following
theorem, which shows that sequences of representations which are centrally stable
are also Specht stable, and thus also monotone in the sense of Church and representation
stable in the sense of Church-Farb.

\begin{maintheorem}
\label{maintheorem:centraltospecht}
Let
$$V_1 \longrightarrow V_2 \longrightarrow V_3 \longrightarrow V_4 \longrightarrow \cdots$$
be a coherent sequence of representations of the symmetric group over a field $\Field$ 
which is centrally stable with stability starting at $N$.  Assume that either $\Char(\Field)=0$
or $\Char(\Field) \geq 2N+2$.  Then the sequence is Specht stable with stability
starting at $2N+1$.
\end{maintheorem}

One special case of Theorem \ref{maintheorem:centraltospecht} appears in the literature.  Assume
that $\Char(\Field)=0$.  Fix some $N \geq 2$ and let $V_N$ be a finite-dimensional representation
of $S_N$.  Recalling that $\Trivial{k}$ is the trivial representation of $S_k$, for $n \geq N$ define
$V_n = \Ind_{S_N \times S_{n-N}}^{S_n} V_N \boxtimes \Trivial{n-N}$.  Here $V_N \boxtimes \Trivial{n-N}$ is
the external tensor product of the $S_N$-representation $V_N$ and the $S_{n-N}$-representation $\Trivial{n-N}$, i.e.\ the
$S_N \times S_{n-N}$-representation whose underlying vector space is $V_N \otimes_{\Field} \Trivial{n-N}$ and
where $(g,h) \in S_N \times S_{n-N}$ acts on $v \otimes w \in V_N \otimes_{\Field} \Trivial{n-N}$ via
$(g,h)(v \otimes w) = (g v) \otimes (h w)$.
There are natural maps
$V_n \hookrightarrow V_{n+1}$, and it is easy to see that the sequence
$$0 \longrightarrow \cdots \longrightarrow 0 \longrightarrow V_N \longrightarrow V_{N+1} \longrightarrow V_{N+2} \longrightarrow \cdots$$
is centrally stable starting at $N$.  Hemmer \cite{HemmerStability} proved that this sequence
is representation stable in the sense of Church-Farb starting at $2N$
and Church \cite{ChurchConfiguration} proved that it is monotone starting at $N$.  There is
also an alternate proof of both of these results due to Sam-Weyman \cite{SamWeyman}.

\paragraph{Outline of paper.}
We begin in \S \ref{section:stabilitymachine} by describing our machine for proving central stability.
This is followed by \S \ref{section:congruence}, which shows how to apply this machine to congruence
subgroups and prove Theorems \ref{maintheorem:congruence} and \ref{maintheorem:congruencesl}.  Next,
in \S \ref{section:chaincomplex} we construct the {\em central stability chain complex}, which is a
technical tool needed for our machine.  In \S \ref{section:machinetheory}, we prove that our machine works.  This proof depends
on a proposition about the central stability chain complex.  This proposition is proven in \S \ref{section:spechtstability},
which also defines Specht stability.  This proof depends on Theorem \ref{maintheorem:centraltospecht},
which is proven in \S \ref{section:centraltospecht}.  Finally, in \S \ref{section:polynomial}
we prove Theorem \ref{maintheorem:polynomial}.

\paragraph{Acknowledgments.}
I wish to thank Ruth Charney, Jordan Ellenberg, Benson Farb, Oscar Randal-Williams,
and Ben Webster for their help.  I want to offer special thanks to Tom Church for pointing
out an error in a previous version of this paper and for his help in figuring
out how to patch it.

\section{Description of central stability machine}
\label{section:stabilitymachine}

We now describe our machine for proving central stability.  This machine 
is similar to the classical homological stability machine as
described in, for example, \cite[\S 5]{HatcherWahl} (which we recommend reading as motivation).

Fix a field $\Field$.  Assume that we are given groups $\{G_n\}_{n=1}^{\infty}$ and
$\{\widetilde{G}_n\}_{n=1}^{\infty}$ together with splittings $\widetilde{G}_n = G_n \rtimes S_n$.  Moreover,
assume that we are given inclusions $G_{n} \hookrightarrow G_{n+1}$ and $\widetilde{G}_n \hookrightarrow \widetilde{G}_{n+1}$
for all $n$ which fit into a commutative diagram of the form
$$\xymatrixrowsep{0.8pc}\xymatrixcolsep{2pc}\xymatrix{
              & 1 \ar[d]              & 1 \ar[d]          & 1 \ar[d]              & \\
\cdots \ar[r] & G_{n-1} \ar[d] \ar[r] & G_n \ar[d] \ar[r] & G_{n+1} \ar[d] \ar[r] & \cdots \\
\cdots \ar[r] & \widetilde{G}_{n-1} \ar[d] \ar[r] & \widetilde{G}_n \ar[d] \ar[r] & \widetilde{G}_{n+1} \ar[d] \ar[r] & \cdots \\
\cdots \ar[r] & S_{n-1} \ar[d] \ar[r] & S_n \ar[d] \ar[r] & S_{n+1} \ar[d] \ar[r] & \cdots \\
	      & 1              & 1          & 1        & }$$
with the maps $S_{n} \rightarrow S_{n+1}$ the standard inclusions.  The conjugation action of $\widetilde{G}_n$
on its normal subgroup $G_n$ induces an action of $S_n$ on $\HH_{k}(G_n;\Field)$ for
all $k \geq 0$.  
Our goal is to prove that the coherent sequence
$$\HH_k(G_1;\Field) \longrightarrow \HH_k(G_2;\Field) \longrightarrow \HH_k(G_3;\Field) \longrightarrow \cdots$$
of representations of the symmetric group is centrally stable.

Before describing the inputs to our machine, we will need a definition.  In this paper, all actions of groups
on simplicial complexes are assumed to be simplicial.

\begin{definition}
A group $G$ acts on a simplicial complex $X$ {\em nicely} if it satisfies the following condition.
Consider two vertices $w$ and $w'$ of $X$ which are joined by an edge.  Then there
does not exist any $g \in G$ such that $g \cdot w = w'$.
\end{definition}

\begin{remark}
If $G$ acts nicely on a simplicial complex $X$, then $X/G$ can be equipped with the structure
of a cell complex whose $\ell$-cells are the $G$-orbits of $\ell$-cells in $X$.  We remark that this
might not be a simplicial complex structure since there might be multiple cells with the same set
of vertices (it is what we will call a {\em weak simplicial complex} in \S \ref{section:coefficients}
below).
\end{remark}

\noindent
For each $n$, our machine will require the following inputs.
\begin{itemize}
\item A simplicial complex $X_n$ such that $G_n$ acts nicely on $X_n$.  Also, this action
should extend to a (not necessarily nice) action of $\widetilde{G}_n$ on $X_n$.
\item An $(n-1)$-simplex $\Delta_n$ of $X_n$ and an enumeration $\{v_1^n,\ldots,v_n^n\}$ of
the vertices of $\Delta_n$.
\end{itemize}
Of course, we will require these inputs to satisfy a sequence of conditions.
First, we will need the $X_n$ to be highly connected so that we can use these actions to 
calculate the homology groups of $G_n$.  In fact, we can get away with assuming that the $X_n$
are highly acyclic.  Recall that a space $Y$ is {\em $k$-acyclic}
if $\RH_q(Y;\Z) = 0$ for $0 \leq q \leq k$.  We make the following assumption about the $X_n$.

\begin{assumption}
\label{assumption:highlyacyclic}
For some constant $C \geq 1$, for all $k \geq 1$ the space $X_n$ is $k$-acyclic for $n \geq \Bound{C 2^{k-1} - 3}$.
\end{assumption}

\noindent
Next, we will need $\Delta_n$ to be a strict fundamental domain for the action, at least
in a stable range.  More precisely, we need the following.

\begin{assumption}
\label{assumption:quotient}
For $k \geq 1$ and $n \geq \Bound{C 2^{k-1}-3}$, the $G_n$-orbit of every simplex in 
the $(k+2)$-skeleton of $X_n$ contains a simplex of $\Delta_n$.  Here $C$ is the same constant as in Assumption
\ref{assumption:highlyacyclic}.
\end{assumption}

\begin{remark}
The niceness of the $G_n$-action on $X_n$ ensures that no two simplices of $\Delta_n$ are in
the same $G_n$-orbit.
\end{remark}

\noindent
We need the following assumption on the stabilizers of this action.

\begin{assumption}
\label{assumption:stabilizer}
For $0 \leq i \leq n-2$, the stabilizer in $G_n$ of the simplex $\{v_{n-i}^n,\ldots,v_n^n\}$ 
is $G_{n-i-1}$.
\end{assumption}

\noindent
Our final two assumptions concern the extension of the action to $\widetilde{G}_n$, and in particular
the action of $S_n \subset \widetilde{G}_n$.

\begin{assumption}
\label{assumption:symmetric}
The action of $S_n$ preserves the set $\{v_1^n,\ldots,v_n^n\}$.  Moreover,
the action of $S_n$ on this set is the usual permutation action.
\end{assumption}

\begin{remark}
Assumptions \ref{assumption:quotient} and \ref{assumption:symmetric} together imply that
$\widetilde{G}_n$ acts transitively on $k+2$-simplices for $n \geq \Bound{C 2^{k-1}-2}$.  
\end{remark}

\begin{assumption}
\label{assumption:centralizer}
Consider any $2 \leq m \leq n$, and denote by $S_{\{m,\ldots,n\}} \subset S_n$ the symmetric group
on the set $\{m,\ldots,n\}$.  Then $S_{\{m,\ldots,n\}}$ lies in the centralizer of $G_{m-1} \subset G_n$.
\end{assumption}

\noindent
With these assumptions, our theorem is as follows.  Its proof is in \S \ref{section:machinetheory}.

\begin{theorem}
\label{theorem:stabilitymachine}
Let the notation and assumptions be as above, and fix some $k \geq 1$.  Assume that
either $\Char(\Field) = 0$ or $\Char(\Field) \geq \Bound{C 2^{k-1} - 3}$.  Then the sequence
$\{\HH_k(G_n;\Field)\}_{n=1}^{\infty}$ of representations of the symmetric group
is centrally stable with stability starting at $\Bound{C 2^{k-1} - 4}$.
\end{theorem}

\section{Stability for congruence subgroups}
\label{section:congruence}

We now show how to apply Theorem \ref{theorem:stabilitymachine} to congruence subgroups and
prove Theorems \ref{maintheorem:congruence} and \ref{maintheorem:congruencesl}.  The proofs
of these results are similar, so for concreteness we will give the details for
Theorem \ref{maintheorem:congruence} and let the reader make the obvious modifications
to prove Theorem \ref{maintheorem:congruencesl}.

\paragraph{Setup.}
Fix a ring $R$ and a proper $2$-sided ideal $q$ of $R$.  Recall that
$$\GL_n(R,q) = \Ker(\GL_n(R) \rightarrow \GL_n(R/q)).$$
Next, let $\widetilde{\GL}_n(R,q) = \GL_n(R,q) \cdot S_n$, where $S_n < \GL_n(R)$ is the
group of permutation matrices.  Clearly $\widetilde{\GL}_n(R,q) = \GL_n(R,q) \rtimes S_n$.

We will assume that $(R,q)$ satisfies the stable range condition $\SR_{d+2}$, which we now define.
See \cite[Chapter 5]{BassKTheory} for more details.  Let $\vec{e}_i \in R^n$ denote the 
vector with a $1$ in position $i$ and zeros elsewhere.

\begin{definition}
A set $\{\vec{v}_1,\ldots,\vec{v}_k\}$ of vectors in $R^n$ is
{\em unimodular} if $R \vec{v}_1 + \cdots + R \vec{v}_k$ is a direct summand of $R^n$.
\end{definition}

\begin{remark}
If $\vec{v} = (a_1,\ldots,a_n) \in R^n$ is a vector, then the set $\{\vec{v}\}$ is unimodular
if and only if $R a_1 + \cdots + R a_n = R$.  We will then say that the vector $\vec{v}$ is unimodular.
\end{remark}

\begin{definition}
We will say that $(R,q)$ satisfies the {\em stable range condition} $\SR_{d+2}$
if the following condition is satisfied for all
$n \geq d+2$.  Let $\vec{v} = (a_1,\ldots,a_n) \in R^n$ be a unimodular vector such that
$\vec{v} \equiv \vec{e}_1$ modulo $q$.  There then exist $b_1,\ldots,b_{n-1} \in q$
such that $(a_1 + b_1 a_n,\ldots,a_{n-1}+b_{n-1} a_n) \in R^{n-1}$ is unimodular.
\end{definition}

\paragraph{The simplicial complexes.}
We now discuss the simplicial complex we will use, which is a slight variant on a complex
introduced by Charney in \cite{CharneyCongruence}.  Let $\cdot$ denote the usual
dot product on $R^n$.  We remark that if $R$ is not commutative, then $\cdot$ is not commutative.

\begin{definition}
The $n$-dimensional {\em complex of split partial bases} over $q$, denoted $\Split_n(R,q)$,
is the simplicial complex whose $k$-simplices are sets
$\{(\vec{v}_0,\vec{w}_0),\ldots,(\vec{v}_k,\vec{w}_k)\} \subset R^n \times R^n$
satisfying the following conditions.
\begin{itemize}
\item The set $\{\vec{v}_0,\ldots,\vec{v}_k\}$ is unimodular.
\item For each $0 \leq i \leq k$, there exists some $1 \leq j_i \leq n$ such that
$\vec{v}_i \equiv \vec{w}_i \equiv \vec{e}_{j_i}$ modulo $q$.
\item For $0 \leq i,j \leq k$, we have $\vec{v}_i \cdot \vec{w}_j = \delta_{ij}$.
\end{itemize}
\end{definition}

\begin{remark}
One should think of a vertex $(\vec{v},\vec{w})$ of $\Split_n(R,q)$ as consisting
of a unimodular vector $\vec{v}$ together with a distinguished splitting
$R^n = \Span{\vec{v}} \oplus W$, where $W = \Set{$\vec{x}$}{$\vec{x} \cdot \vec{w} = 0$}$.
\end{remark}

The group $\GL_n(R,q)$ acts on $\Split_n(R,q)$ via the formula
$$M \cdot \{(\vec{v}_0,\vec{w}_0),\ldots,(\vec{v}_k,\vec{w}_k)\} = \{(M \vec{v}_0, (M^{-1})^t \vec{w}_0),\ldots, (M \vec{v}_k,(M^{-1})^t \vec{w}_0)\}.$$
This action is clearly nice and extends over $\widetilde{\GL}_n(R,q)$.

\paragraph{The distinguished simplex.}
Our distinguished simplex in $\Split_n(R,q)$ will be 
$\Delta_n = \{(\vec{e}_1,\vec{e}_1),\ldots,(\vec{e}_n,\vec{e}_n)\}$.

\paragraph{Verification of the assumptions.}
We now verify the five assumptions from \S \ref{section:stabilitymachine}.  Our constant
$C$ will be $d+8$.  Theorem \ref{maintheorem:congruence} will then follow
from Theorem \ref{theorem:stabilitymachine}. 
Assumptions 3--5 are trivial, so we will omit the details of their verification.
It remains to verify Assumptions 1 and 2.

\paragraph{Assumption 1.}
This assumption says that $\Split_n(R,q)$ is $k$-acyclic for $n \geq (d+8) 2^{k-1} - 3$.
In fact, we have the following.

\begin{lemma}
If $(R,q)$ satisfies $\SR_{d+2}$, then the complex $\Split_n(R,q)$ is $\frac{n-d-3}{2}$-acyclic.
\end{lemma}
\begin{proof}
Let $\Poset(\Split_n(R,q))$ be the face poset of $\Split_n(R,q)$, i.e.\ the poset
whose elements are simplices of $\Split_n(R,q)$ and where $\sigma \leq \sigma'$ if
$\sigma$ is a face of $\sigma'$.  The geometric realization $|\Poset(\Split_n(R,q))|$ of
$\Poset(\Split_n(R,q))$ is the barycentric subdivision of $\Split_n(R,q)$.  Next,
define $\Poset'(\Split_n(R,q))$ to be the poset whose elements are ordered
sequences $(x_0,\ldots,x_k)$ of distinct vertices of $\Split_n(R,q)$ such that the unordered
set $\{x_0,\ldots,x_k\}$ is a simplex of $\Split_n(R,q)$.  Sequences $s$ and $s'$ in
$\Poset'(\Split_n(R,q))$ satisfy $s \leq s'$ if $s$ is a subsequence of $s'$.  In
\cite[Theorem 3.5]{CharneyCongruence}, Charney proved that the geometric realization
of $\Poset'(\Split_n(R,q))$ is $\frac{n-d-3}{2}$-acyclic.

There is a natural map $\pi : \Poset'(\Split_n(R,q)) \rightarrow \Poset(\Split_n(R,q))$ which
``forgets'' the ordering on a sequence.  Choose a total ordering on the vertices
of $\Split_n(R,q)$, and define a poset map $\rho : \Poset(\Split_n(R,q)) \rightarrow \Poset'(\Split_n(R,q))$
by the formula $\rho(\{x_0,\ldots,x_k\}) = (x_0,\ldots,x_k)$, where the ordering on the $x_i$
is chosen such that $x_0 < x_1 < \cdots < x_k$.  It is clear that $\pi \circ \rho = 1$, which
implies that the map on geometric realizations induced by $\pi$ is surjective on reduced homology.
We conclude that the geometric realization of $\Poset(\Split_n(R,q))$ is $\frac{n-d-3}{2}$-acyclic.
\end{proof}

\paragraph{Assumption 2.}
This assumption says that all the $\GL_n(R,q)$-orbits of $(k+2)$-simplices of
$\Split_n(R,q)$ contain simplices in $\Delta_n$ for $n \geq (d+8) 2^{k-1} - 3$.  
This is an immediate consequence of the following.

\begin{lemma}
If $(R,q)$ satisfies $\SR_{d+2}$, then the group $\widetilde{\GL}_n(R,q)$ acts transitively
on $k$-simplices of $\Split_n(R,q)$ for $k \leq n-d$.
\end{lemma}

\noindent
This lemma can be proven exactly like \cite[Proposition on p.\ 2101]{CharneyCongruence}.

\section{The central stability chain complex}
\label{section:chaincomplex}

In this section, we introduce the {\em central stability chain complex}, which will play
a key role in our proofs.  We need a definition and some notation.

\begin{notation}
If $Y$ is a set, then let $S_Y$ denote the symmetric group on $Y$.
\end{notation}

\begin{definition}
Let
\begin{equation}
\label{eqn:potential}
V_n \stackrel{\phi_{n}}{\longrightarrow} V_{n+1} \stackrel{\phi_{n+1}}{\longrightarrow} V_{n+2} \stackrel{\phi_{n+2}}{\longrightarrow} \cdots \stackrel{\phi_{m-1}}{\longrightarrow} V_m
\end{equation}
be a sequence of maps between vector spaces, where $V_i$ is a representation of $S_i$ and
$\phi_i$ is $S_i$-equivariant for all $i$.  This sequence is {\em potentially centrally stable} if
the following holds for all $n \leq i < j \leq m$.  Let $\vec{v} \in V_j$ be in the image of $V_i$.  Then
$\sigma \cdot \vec{v} = \vec{v}$ for all $\sigma \in S_{\{i+1,\ldots,j\}}$.
\end{definition}

\begin{remark}
If $V_{i+1} = \Stab{V_i}{V_{i-1}}$ for all $n < i < m$, then the sequence is obviously
potentially centrally stable.
\end{remark}

\begin{notation}
For $k \geq 0$, let $\Alternate{k}$ denote the sign representation of $S_k$.  Also,
if $V$ is a representation of $S_n$, then denote $\Ind_{S_n \times S_k}^{S_{n+k}} V \boxtimes \Alternate{k}$
by $\Induce{V}{k}$.  We remark that we discussed the external tensor product $\boxtimes$ at the end
of the introduction.
\end{notation}

\noindent
Given a potentially centrally stable sequence like \eqref{eqn:potential} and some $M \geq m$, 
the associated $M$-central
stability chain complex will be of the form
$$\Induce{V_n}{M-n} \rightarrow \Induce{V_{n-1}}{M-n-1} \rightarrow \cdots \rightarrow \Induce{V_{m}}{M-m}.$$
Observe that each vector space is a representation of $S_M$.

\paragraph{Boundary maps.}
We begin by constructing the appropriate boundary maps.  Let $\phi_n : V_{n} \rightarrow V_{n+1}$ be an $S_n$-equivariant
map from a representation of $S_n$ to a representation of $S_{n+1}$, and fix some $M > n$.  We
will construct an $S_M$-equivariant map $\partial_{n} : \Induce{V_n}{M-n} \rightarrow \Induce{V_{n+1}}{M-n-1}$.
This map will be called the {\em $M$-boundary map} associated to $\phi_n$.

The construction is as follows.  
Let $C_{n}$ be a set of right coset representatives for $S_{\{n+2,\ldots,M\}}$ in $S_{\{n+1,\ldots,M\}}$, and
define an $S_n$-equivariant map $\partial_{n}' : V_n \rightarrow \Induce{V_{n+1}}{M-n-1}$ via the formula
$$\partial_{n}'(\vec{v}) = \sum_{\sigma \in C_{n}} (-1)^{|\sigma|} \sigma \cdot \phi_n(\vec{v}) \quad \quad (\vec{v} \in V_n).$$
Here we are identifying $V_{n+1}$ with its image in $\Induce{V_{n+1}}{M-n-1}$.  

\begin{lemma}
\label{lemma:cnmindependence}
The map $\partial_{n}'$ does not depend on the choice of $C_{n}$.
\end{lemma}
\begin{proof}
Let $C_{n}'$ be another set of right coset representatives.  For $s \in C_{n}'$, there
exists a unique $\sigma_{s} \in C_{n}$ and $\tau_{s} \in S_{\{n+2,\ldots,M\}}$ such
that $s = \sigma_{s} \tau_{s}$.  For $\vec{v} \in V_{n}$, we then have
$$(-1)^{|s|} s \cdot \phi_n(\vec{v}) = (-1)^{|\sigma_s \tau_s|} \sigma_s \tau_s \cdot \vec{v}
= (-1)^{|\sigma_s|} (-1)^{|\tau_s|} (-1)^{|\tau_s|} \sigma_s \cdot \vec{v} = (-1)^{|\sigma_s|} \sigma_s \cdot \vec{v}.$$
The lemma follows.
\end{proof}

\begin{lemma}
\label{lemma:equivariance}
For $\delta \in S_{\{n+1,\ldots,M\}}$ and $\vec{v} \in V_n$, we have 
$\delta \cdot \partial_{n}'(\vec{v}) = (-1)^{|\delta|} \partial_{n}'(\vec{v})$.
\end{lemma}
\begin{proof}
For $\sigma \in C_{n}$, there exists some $\sigma_{\delta} \in C_{n}$ 
and $\tau_{\sigma,\delta} \in S_{\{n+2,\ldots,M\}}$ such that
$\delta \sigma = \sigma_{\delta} \tau_{\sigma,\delta}$.  We have 
$(-1)^{|\delta|} (-1)^{|\sigma|} = (-1)^{|\sigma_{\delta}|} (-1)^{|\tau_{\sigma,\delta}|}$, and
thus
\begin{align*}
\delta \cdot \partial_{n}'(\vec{v}) &= \sum_{\sigma \in C_{n}} (-1)^{|\sigma|} \sigma_{\delta} \tau_{\sigma,\delta} \cdot \phi_{n}(\vec{v}) \\
&= \sum_{\sigma \in C_{n}} (-1)^{|\sigma|} (-1)^{|\tau_{\sigma,\delta}|} \sigma_{\delta} \cdot \phi_{n}(\vec{v}) \\
&= (-1)^{|\delta|} \sum_{\sigma \in C_{n}} (-1)^{|\sigma_{\delta}|} \sigma_{\delta} \cdot \phi_{n}(\vec{v}) = (-1)^{|\delta|} \partial_{n+i}'(\vec{v}).
\end{align*}
The final equality follows from the fact that the map $\sigma \mapsto \sigma_{\delta}$ is a permutation
of $C_{n}$.  
\end{proof}

\noindent
Completing our construction of $\partial_n$, Lemma \ref{lemma:equivariance} implies that 
$\partial_{n}'$ induces an $S_n \times S_{M-n}$-equivariant
map $V_{n} \boxtimes \Alternate{M-n} \rightarrow \Induce{V_{n+1}}{M-n-1}$, so
we obtain a $S_M$-equivariant map $\partial_{n} : \Induce{V_n}{M-n} \rightarrow \Induce{V_{n+1}}{M-n-1}$.

\paragraph{First relation to central stabilization.}
We pause now to observe that the above gives a presentation for the central stabilization of a map.

\begin{lemma}
\label{lemma:centralstab}
Let $\phi_n : V_n \rightarrow V_{n+1}$ be an $S_n$-equivariant map from a representation
of $S_n$ to a representation of $S_{n+1}$ and let $\partial_n$ be the $(n+2)$-boundary map
induced by $\phi_n$.  Set $V_{n+2} = \Stabb{V_n}{V_{n+1}}{\phi_n}$.  There is then an exact sequence
$$\Induce{V_n}{2} \stackrel{\partial_{n}}{\longrightarrow} \Induce{V_{n+1}}{1} \longrightarrow V_{n+2} \longrightarrow 0.$$
\end{lemma}
\begin{proof}
By definition, there is a surjection $\pi : \Induce{V_n+1}{1} \rightarrow V_{n+2}$.  Let
$i : V_{n} \rightarrow \Induce{V_{n+1}}{1}$ be the $S_{n}$-equivariant map obtained
by composing $V_{n} \rightarrow V_{n+1}$ with the natural inclusion $V_{n+1} \hookrightarrow \Induce{V_{n+1}}{1}$.
There is an $S_{n} \times S_2$-equivariant map $j:V_{n} \boxtimes \Alternate{2} \rightarrow \Induce{V_{n+1}}{1}$
defined by $j(v) = v - (n,n+1) \cdot v$.  By the universal property of the induced representation,
this extends to a map
$\rho : \Induce{V_{n}}{2} \rightarrow \Induce{V_{n+1}}{1}$.  It is easy to see that $\rho$ is exactly the $(n+2)$-boundary
map associated to $\phi_n$.  By definition, the image of $\rho$ is the kernel of $\pi$, and we are done.
\end{proof}

\paragraph{The chain complex.}
We now prove that the above gives a chain complex, which as we said we will call the {\em $M$-central stability chain complex}.

\begin{lemma}
\label{lemma:chaincomplex}
Let
$$V_n \stackrel{\phi_{n}}{\longrightarrow} V_{n+1} \stackrel{\phi_{n+1}}{\longrightarrow} V_{n+2} \stackrel{\phi_{n+2}}{\longrightarrow} \cdots \stackrel{\phi_{m-1}}{\longrightarrow} V_m$$
be a potentially centrally stable sequence of representations of the symmetric group and let $M \geq m$.  
For $n \leq i < m$, let $\partial_i$ be the $M$-boundary map associated to $\phi_i$.  Then the sequence
$$\Induce{V_n}{M-n} \stackrel{\partial_n}{\longrightarrow} \Induce{V_{n+1}}{M-n-1} \stackrel{\partial_{n+1}}{\longrightarrow} \cdots \stackrel{\partial_{m-1}}{\longrightarrow} \Induce{V_{m}}{M-m} \longrightarrow 0$$
of representations of $S_M$ is a chain complex.
\end{lemma}
\begin{proof}
Throughout this proof, we will regard $V_i$ as a subspace of $\Induce{V_i}{M-i}$ for all
$n \leq i \leq m$.  Fix some $n \leq i < m-2$, and consider $\vec{v} \in V_{i}$.  It is enough
to prove that $\partial_{i+1}(\partial_i(\vec{v})) = 0$.  Let $C_i$ and $C_{i+1}$ be 
the sets of coset representatives used to construct $\partial_i$ and $\partial_{i+1}$.  Set
$\vec{w} = \phi_{i+1}(\phi_{i}(\vec{v}))$.
Observe that $\partial_{i+1}(\partial_i(\vec{v}))$ equals
\begin{align*}
\partial_{i+1}(\sum_{\sigma \in C_{i}} (-1)^{|\sigma|} \sigma \cdot \phi_{i}(\vec{v})) 
= \sum_{\sigma \in C_{i}, \sigma' \in C_{i+1}} (-1)^{|\sigma \sigma'|} \sigma \sigma' \cdot \vec{w}.
\end{align*}
The set $\Set{$\sigma \sigma'$}{$\sigma \in C_{i}, \sigma' \in C_{i+1}$}$ is a set of 
right coset representatives for $S_{\{i+3,\ldots,M\}}$ in $S_{\{i+1,\ldots,M\}}$.

Let $D$ be a set of right coset representatives for $S_{\{i+1,i+2\}} \times S_{\{i+3,\ldots,M\}}$ in 
$S_{\{i+1,\ldots,M\}}$.  The set
$\Set{$\sigma$}{$\sigma \in D$} \cup \Set{$\sigma (i+1,i+2)$}{$\sigma \in D$}$ is 
thus a set of right coset representatives for $S_{\{i+3,\ldots,M\}}$ in
$S_{\{i+1,\ldots,M\}}$.  
By an argument similar to that in the proof of Lemma \ref{lemma:cnmindependence}, 
we deduce that $\partial_{i+1}(\partial_{i}(\vec{v}))$ equals
$$\sum_{\sigma \in D} \left((-1)^{|\sigma|} \sigma \cdot \vec{w} + (-1)^{|\sigma|+1} \sigma (i+1,i+2) \cdot \vec{w}\right) = 0.$$
Here we have used the fact that $(i+1,i+2) \cdot \vec{w} = \vec{w}$, 
which follows from the potential central stability of our sequence.
\end{proof}

\paragraph{Exactness.}
The following proposition is perhaps the most important technical result in this paper.  Its proof is contained in
\S \ref{section:spechtstability}--\ref{section:centraltospecht}.  We postpone it because it uses more representation
theory than the rest of the paper, and we want to separate as much as possible the representation theoretic parts
of this paper from the topological parts.  See the beginning of \S \ref{section:spechtstability} for a road map
of its proof.

\begin{proposition}
\label{proposition:exactness}
Let 
$$V_1 \rightarrow V_2 \rightarrow \cdots$$
be a coherent sequence of representations over a field $\Field$ of the symmetric group which is centrally stable starting
at $N$.  Assume that either $\Char(\Field)=0$ or $\Char(\Field) \geq 2N+2$.  
Consider $n$ and $m$ and $M$ such that $2N+1 \leq n \leq m \leq M$.  Then the $M$-central stability chain
complex associated to the potentially centrally stable sequence
$$V_n \rightarrow V_{n+1} \rightarrow \cdots \rightarrow V_m$$
is exact.
\end{proposition}

\section{Proof that the central stability machine works}
\label{section:machinetheory}

In this section, we prove Theorem \ref{theorem:stabilitymachine}.  The actual
proof is in \S \ref{section:stabilitymachineproof}.  This is preceded by two
sections containing necessary background : \S \ref{section:coefficients}
discusses coefficient systems and \S \ref{section:equivariant} discusses some basic results in equivariant
homology theory.

\subsection{Coefficient systems}
\label{section:coefficients}

Fix a field $\Field$.  For technical reasons, we will need to work in the
category of {\em weak simplicial complexes}, which are defined exactly
like simplicial complexes except that they can have more than one simplex
spanned by a single set of vertices.  Fix a weak simplicial complex $X$.  Observe that
the simplices of $X$ form the objects of a category with a unique morphism 
$\sigma' \rightarrow \sigma$ whenever $\sigma'$ is a face of $\sigma$.

\begin{definition}
A {\em coefficient system} on $X$ is a contravariant functor from the category
associated to $X$ to the category of vector spaces over $\Field$.
\end{definition}

\begin{remark}
In other words, a coefficient system $\mathcal{F}$ on $X$ consists of $\Field$-vector spaces $\mathcal{F}(\sigma)$ for simplices
$\sigma$ of $X$ and linear maps $\mathcal{F}(\sigma' \rightarrow \sigma) : \mathcal{F}(\sigma) \rightarrow \mathcal{F}(\sigma')$ whenever $\sigma'$ is a
face of $\sigma$.  These linear maps must satisfy the obvious compatibility condition.
\end{remark}

\begin{definition}
Let $\mathcal{F}$ be a coefficient system on $X$.  Fix a total ordering on the elements of $X^{(0)}$.  The
{\it simplicial chain complex} of $X$ with coefficients in $\mathcal{F}$ is as follows.  Define 
$$\Chain_k(X;\mathcal{F}) = \bigoplus_{\sigma \in X^{(k)}} \mathcal{F}(\sigma).$$
Next, define a differential $\partial : \Chain_k(C;\mathcal{F}) \rightarrow \Chain_{k-1}(C;\mathcal{F})$
in the following way.  Consider $\sigma \in X^{(k)}$.  
We will denote an element of $\mathcal{F}(\sigma) \subset \Chain_k(X;\mathcal{F})$
by $c \cdot \sigma$ for $c \in \mathcal{F}(\sigma)$.  Let $v_0,\ldots,v_k$ be the vertices of $\sigma$.  
Choose the ordering
such that $v_i < v_{i+1}$ for $0 \leq i < k$.  Denote
by $\sigma_i$ the face of $\sigma$ opposite the vertex $v_i$.  For $c \in \mathcal{F}(\sigma)$, we then define
$$\partial(c \cdot \sigma) = \sum_{i=0}^k (-1)^i c_i \cdot \sigma_i,$$
where $c_i$ is the image of $c$ under the morphism
$\mathcal{F}(\sigma' \rightarrow \sigma) : \mathcal{F}(\sigma) \longrightarrow \mathcal{F}(\sigma_i)$.
Taking the homology of $\Chain_{\ast}(X;\mathcal{F})$ yields the {\it homology groups of $X$ with coefficients in $\mathcal{F}$},
which we will denote by $\HH_{\ast}(X;\mathcal{F})$.
\end{definition}

\begin{remark}
If $V$ is an $\Field$-vector space and $\mathcal{F}$ is the coefficient system that assigns $V$ to every simplex and the identity map to every
face map, then $\HH_{\ast}(X;\mathcal{F}) \cong \HH_{\ast}(X;V)$.  We will call this a {\em constant system of coefficients}.
\end{remark}

\subsection{Equivariant homology}
\label{section:equivariant}

We will need a small portion of the theory of equivariant homology.  All
the results below are contained (implicitly or explicitly) in \cite[\S VII]{BrownCohomology}.
Recall that if $G$ acts nicely on a simplicial complex $X$, then $X/G$ is a weak
simplicial complex in a natural way.

\begin{definition}
Consider a group $G$ acting nicely on a simplicial complex $X$.
Let $EG$ be a contractible simplicial complex on which $G$ acts nicely and freely, so $EG / G$ is a classifying space for $G$.  
Define $EG \times_G X$ to be the quotient of $EG \times X$ by the diagonal action of $G$.  The {\it $G$-equivariant
homology groups of $X$}, denoted $H_\ast^G(X;\Field)$, are defined to be $H_\ast(EG \times_G X;\Field)$.
\end{definition}

\begin{remark}
It is easy to see that $\HH_{\ast}^G(X;\Field)$ does not depend on the choice of $EG$.  The construction
of $EG \times_G X$ is known as the {\em Borel construction}.
\end{remark}

The following lemma summarizes two key properties of these homology groups.

\begin{lemma}
\label{lemma:equivarianthomology}
Consider a group $G$ acting nicely on a simplicial complex $X$.
\begin{itemize}
\item There is a canonical map $\HH_{\ast}^G(X;\Field) \rightarrow \HH_{\ast}(G;\Field)$.
\item If $X$ is $k$-acyclic, then the map $\HH_i^G(X;\Field) \rightarrow \HH_i(G;\Field)$ is an
isomorphism for $i \leq k$.
\end{itemize}
\end{lemma}

\begin{remark}
The map $\HH_{\ast}^G(X;\Field) \rightarrow \HH_{\ast}(G;\Field)$ comes from map $EG \times_G X \rightarrow EG/G$
induced by the projection of $EG \times X$ onto its first factor.  The second claim is an immediate consequence
of the spectral sequence whose $E^2$ page is (7.2) in \cite[\S VII.7]{BrownCohomology}.
\end{remark}

To calculate equivariant homology groups, we will need a certain spectral sequence.  First, a definition.

\begin{definition}
Consider a group $G$ acting nicely on a simplicial complex $X$.  Define a coefficient
system $\mathcal{H}_q(G,X;\Field)$ on $X/G$ as follows.  Consider a simplex $\sigma$ of $X/G$.  
Let $\widetilde{\sigma}$ be any lift of $\sigma$ to $X$.  Set
$$\mathcal{H}_q(G,X;\Field)(\sigma) = \HH_q(G_{\widetilde{\sigma}};\Field),$$
where $G_{\widetilde{\sigma}}$ is the stabilizer of $\widetilde{\sigma}$.  It is easy to see that this does not depend on
the choice of $\widetilde{\sigma}$ and that it defines a coefficient system on $X/G$.
\end{definition}

\noindent
Our spectral sequence is then as follows.  It can be easily extracted from \cite[\S VII.8]{BrownCohomology}

\begin{theorem}
\label{theorem:mainspectralsequence}
Let $G$ be a group acting nicely on a simplicial complex $X$.  There is then a spectral sequence
converging to $\HH_{\ast}^G(X;\Field)$ with 
$$E^2_{p,q} \cong \HH_p(X/G;\mathcal{H}_q(G,X;\Field)).$$
\end{theorem}

Assume now that $G$ acts nicely on a simplicial complex $X$ which is $k$-connected, and consider
$v \in X^{(0)}$.  The inclusion map $G_v \hookrightarrow G$ induces a map 
$\HH_k(G_v;\Field) \rightarrow \HH_k(G;\Field)$
which is easily described in terms of the spectral sequence in Theorem \ref{theorem:mainspectralsequence}.
First, Lemma \ref{lemma:equivarianthomology} says that the spectral sequence in Theorem
\ref{theorem:mainspectralsequence} converges to $\HH_i(G;\Field)$ for $0 \leq i \leq k$.  Next,
observe that there is a natural map 
$$\HH_k(G_v;\Field) \rightarrow E^2_{0,k} = \HH_0(X/G;\mathcal{H}_k(G,X;\Field))$$
obtained as the composition
$$\HH_k(G_v;\Field) \hookrightarrow \Chain_0(X/G;\mathcal{H}_k(G,X;\Field)) \rightarrow \HH_0(X/G;\mathcal{H}_k(G,X;\Field)),$$
where the first map is the natural inclusion.  The map $\HH_k(G_v;\Field) \rightarrow \HH_k(G;\Field)$
is then the composition
$$\HH_k(G_v;\Field) \rightarrow E^2_{0,k} \rightarrow E^{\infty}_{0,k} \rightarrow \HH_k(G;\Field).$$

\subsection{The proof of Theorem \ref{theorem:stabilitymachine}}
\label{section:stabilitymachineproof}

We now prove Theorem \ref{theorem:stabilitymachine}.  This requires the following standard
lemma.

\begin{lemma}[{\cite[Proposition III.5.3]{BrownCohomology}}]
\label{lemma:inducedrep}
Let $G$ be a group with a subgroup $H$.  Let $V$ be a representation of $G$ and $W \subset V$ be an $H$-subrepresentation.
Picking a set $\{g_i\}_{i \in I}$ of left coset representatives for $H$ in $G$, assume that
$$V = \bigoplus_{i \in I} g_i \cdot W.$$
Then $V \cong \Ind_H^G W$.
\end{lemma}

\begin{proof}[{Proof of Theorem \ref{theorem:stabilitymachine}}]
Let $\{G_i\}$ and $\{\widetilde{G}_i\}$ and $\{X_i\}$ and $\Delta_i = \{v_1^i,\ldots,v_i^i\}$ 
be as in \S \ref{section:stabilitymachine}.  Fix $k \geq 1$.  We wish to prove that the sequence 
$$\HH_k(G_1;\Field) \longrightarrow \HH_k(G_2;\Field) \longrightarrow \HH_k(G_3;\Field) \longrightarrow \cdots$$
of representations of the symmetric group is centrally stable with stability starting at
$\Bound{C 2^{k-1} - 4}$.  Assume as an inductive hypothesis that this is true for all smaller
nonnegative $k$ (for $k=1$, this assumption is vacuous).  To simplify our notation, we will omit the coefficients $\Field$ from our
homology groups and chain groups.

Fix $n \geq \Bound{C 2^{k-1} - 4}$.  We want to prove that there is an $S_{n+1}$-equivariant isomorphism
$$\HH_k(G_{n+1}) \cong \Stab{\HH_k(G_{n-1})}{\HH_k(G_n)}$$
and that the map $\HH_k(G_n) \rightarrow \HH_k(G_{n+1})$ is as in the definition of
central stabilization.  To do this, we will use the spectral sequence from
Theorem \ref{theorem:mainspectralsequence} for the action of $G_{n+1}$ on $X_{n+1}$.
This spectral sequence converges to $\HH_{\ast}^{G_{n+1}}(X_{n+1})$.  Since 
$n+1 \geq \Bound{C 2^{k-1} - 3}$, Assumption \ref{assumption:highlyacyclic} 
implies that $X_{n+1}$ is $k$-acyclic, so
by Lemma \ref{lemma:equivarianthomology} we have $\HH_i^{G_{n+1}}(X_{n+1}) \cong \HH_i(G_{n+1})$ 
for $0 \leq i \leq k$.  

To simplify our notation, we will denote $\Chain_j(X_{n+1}/G_{n+1};\mathcal{H}_i(G_{n+1},X_{n+1}))$ 
by $\Chain_j^i$.  The action of $\widetilde{G}_{n+1}$ on $X_{n+1}$ induces an action of $S_{n+1}$ on
$\Chain_j^i$ which commutes with the boundary map $\Chain_j^i \rightarrow \Chain_{j-1}^i$ for all $i$ and $j$.
The following observation is the key to our proof.

\BeginClaims
\begin{claims}
Fix some $1 \leq i \leq k$.  For all $0 \leq j \leq k+2$, there exists an $S_{n+1}$-equivariant isomorphism
$\eta_j : \Induce{\HH_i(G_{n-j})}{j+1} \rightarrow \Chain_j^i$ such that the diagram
$$\xymatrix{
\Induce{\HH_i(G_{n-k-2}}{k+3} \ar[r] \ar[d]^{\eta_{k+2}} & \Induce{\HH_i(G_{n-k-1})}{k+2} \ar[r] \ar[d]^{\eta_{k+1}} & \cdots \ar[r] & \Induce{\HH_i(G_{n})}{1} \ar[d]^{\eta_0} \\
\Chain_{k+2}^i \ar[r]                                    & \Chain_{k+1}^i \ar[r]                                     & \cdots \ar[r] & \Chain_{0}^i                             }$$
commutes.  Here bottom row is the chain complex computing $\HH_{\ast}(X_{n+1}/G_{n+1};\mathcal{H}_i(G_{n+1},X_{n+1}))$ and the top row
is the $(n+1)$-central stability chain complex for the potentially centrally stable sequence
$$\HH_i(G_{n-k-2}) \rightarrow \HH_i(G_{n-k-1}) \rightarrow \cdots \rightarrow \HH_i(G_{n}).$$
\end{claims}
\begin{proof}[Proof of claim]
Since $G_{n+1}$ acts nicely on $X_{n+1}$, the simplex $\Delta_{n+1}$ of $X_{n+1}$ injects
into $X_{n+1} / G_{n+1}$.  Let $\overline{\Delta}_{n+1}$ be its image.  Assumption \ref{assumption:symmetric}
says that the action of $S_{n+1}$ on $X_{n+1}/G_{n+1}$ preserves $\overline{\Delta}_{n+1}$.  Letting
$$D_j^i = \Chain_j(\overline{\Delta}_{n+1};\mathcal{H}_i(G_{n+1},X_{n+1})),$$
we see that $D_j^i$ is an $S_{n+1}$-representation.  There is a natural map
$\kappa_j : D_j^i \rightarrow \Chain_j^i$ induced by the 
inclusion $\overline{\Delta}_{n+1} \hookrightarrow X_{n+1}/G_{n+1}$.
Since $n+1 \geq \Bound{C 2^{k-1} - 3}$, Assumption \ref{assumption:quotient} implies 
that $\kappa_j$ is an isomorphism for $0 \leq j \leq k+2$.  We will prove that 
\begin{equation}
\label{eqn:dind}
D_j^i \cong \Induce{\HH_i(G_{n-j})}{j+1}
\end{equation}
as $S_{n+1}$-representations.  This will give us
the desired maps $\eta_j$; proving that the indicated diagram commutes is then an easy
exercise in the definitions of the various maps, and is thus omitted.

It remains to prove \eqref{eqn:dind}.  We want to apply Lemma \ref{lemma:inducedrep}.  By definition, we have
\begin{equation}
\label{eqn:ddecomp}
D_j^i = \bigoplus_{1 \leq \ell_0 < \cdots < \ell_j \leq n+1} \HH_i((G_{n+1})_{\{v_{\ell_0}^{n+1},\ldots,v_{\ell_j}^{n+1}\}}).
\end{equation}
By Assumption \ref{assumption:centralizer}, the subgroup $S_{n-j} \times S_{j+1} \subset S_{n+1}$ preserves
the term $\HH_i((G_{n+1})_{\{v_{n+1-j}^{n+1},\ldots,v_{n+1}^{n+1}\}})$ of \eqref{eqn:ddecomp}.  Moreover, Assumption
\ref{assumption:stabilizer} says that $(G_{n+1})_{\{v_{n+1-j}^{n+1},\ldots,v_{n+1}^{n+1}\}} = G_{n-j}$, so as an $S_{n-j} \times S_{j+1}$
representation this term is isomorphic to $(\HH_i(G_{n-j})) \boxtimes \Alternate{j+1}$ (the subgroup
$S_{j+1}$ acts via the sign representation since it is just changing the orientation of the associated
simplex).  The left cosets of $S_{n-j} \times S_{j+1}$ in $S_{n+1}$ are exactly determined by what they
do to the unordered set $\{n+1-j,\ldots,n+1\}$, so letting $C$ be a complete set of such coset
representatives, we obtain that \eqref{eqn:ddecomp} can be rewritten
$$D_j^i = \bigoplus_{\sigma \in C} \sigma \cdot (\HH_i((G_{n+1})_{\{v_{n+1-j}^{n+1},\ldots,v_{n+1}^{n+1}\}})).$$
Lemma \ref{lemma:inducedrep} then implies that $D_j^i \cong \Induce{\HH_i(G_{n-j})}{j+1}$, as desired.
\end{proof}

We can now analyze the $E^2$-page of our spectral sequence.
\begin{claims}
$E^2_{0,k} = \Stab{\HH_k(G_{n-1})}{\HH_k(G_n)}$.
\end{claims}
\begin{proof}[Proof of claim]
Claim 1 implies that
$$E^2_{0,k} = \Coker(\Chain_1^k \rightarrow \Chain_0^k) = \Coker(\Induce{\HH_k(G_{n-1})}{2} \rightarrow \Induce{\HH_k(G_n)}{1}),$$ 
which by Lemma \ref{lemma:centralstab} equals $\Stab{\HH_k(G_{n-1})}{\HH_k(G_n)}$.
\end{proof}

\begin{claims}
$E^2_{j,i} = 0$ for $1 \leq i < k$ and $j \geq 1$ such that $j \leq k-i+1$.  
\end{claims}
\begin{proof}[Proof of claim]
This is asserting that the sequence
$$\Chain^i_{k-i+2} \rightarrow \Chain^i_{k-i+1} \rightarrow \cdots \rightarrow \Chain^i_0$$
is exact.  By Claim 1, this is equivalent to the exactness of the sequence 
\begin{equation}
\label{eqn:inductionseq}
\Induce{\HH_i(G_{n+i-k-2})}{k-i+3} \rightarrow \Induce{\HH_i(G_{n+i-k-1})}{k-i+2} \rightarrow \cdots \rightarrow \Induce{\HH_i(G_n)}{1}
\end{equation}
By induction, the sequence
$$\HH_i(G_1) \rightarrow \HH_i(G_2) \rightarrow \cdots$$
is centrally stable starting at $C 2^{i-1}-4$.  Using the easily-verified inequality $2^a-2^b \geq a-b$ for
integers $a \geq b \geq 0$, we have
\begin{align*}
2(C 2^{i-1} - 4)+1 &= (C 2^{k-1} - 4) - C(2^{k-1} - 2^{i}) - 3\\
                   &\leq (C 2^{k-1} - 4) - ((k-1) - i) - 3 \leq n+i-k-2.
\end{align*}
Proposition \ref{proposition:exactness} therefore implies that \eqref{eqn:inductionseq} is exact,
as desired.
\end{proof}

\begin{claims}
$E^2_{j,0} = 0$ for $1 \leq j \leq k+1$.
\end{claims}
\begin{proof}[Proof of claim]
The coefficient system $\mathcal{H}_0(G_{n+1},X_{n+1})$ is the constant coefficient system $\Field$.  As in
Claim 1, let $\overline{\Delta}_{n+1}$ be the image of $\Delta_{n+1}$ in $X_{n+1}/G_{n+1}$.
Assumption \ref{assumption:quotient} implies $\overline{\Delta}_{n+1}$ contains the entire
$(k+2)$-skeleton of $X_{n+1}/G_{n+1}$.  Since $\overline{\Delta}_{n+1}$ is contractible,
it follows that $\HH_j(X_{n+1}/G_{n+1};\mathcal{H}_0(G_{n+1},X_{n+1}))=0$ for
$1 \leq j \leq k+1$, as desired.
\end{proof}

Summarizing, the part of the $E^2$-page of our spectral sequence needed to 
compute $\HH_k(G_{n+1})$ is as follows.  
\begin{center}
\begin{tabular}{|c@{\hspace{0.2 in}}c@{\hspace{0.2 in}}c@{\hspace{0.2 in}}c@{\hspace{0.2 in}}c@{\hspace{0.2in}}c}
\footnotesize{$\Stab{\HH_k(G_{n-1})}{\HH_k(G_n)}$}      &                         &                         &                         &  & \\
\footnotesize{$\ast$}   & \footnotesize{$0$}      & \footnotesize{$0$}      &                         &  & \\
\footnotesize{$\ast$}   & \footnotesize{$0$}      & \footnotesize{$0$}      & \footnotesize{$0$}      &  & \\
\footnotesize{$\vdots$} & \footnotesize{$\vdots$} & \footnotesize{$\vdots$} & \footnotesize{$\vdots$} & \footnotesize{$\ddots$} & \\
\footnotesize{$\ast$}   & \footnotesize{$0$}      & \footnotesize{$0$}      & \footnotesize{$0$}      & \footnotesize{$\cdots$} & \footnotesize{$0$} \\
\cline{1-6}
\end{tabular}
\end{center}
We conclude that $\HH_k(G_{n+1}) = \Stab{\HH_k(G_{n-1})}{\HH_k(G_n)}$.  The fact that
the map $\HH_k(G_n) \rightarrow \HH_k(G_{n+1})$ is as in the definition of central stability
follows easily from the discussion after Theorem \ref{theorem:mainspectralsequence} together
with the fact that $G_n = (G_{n+1})_{v_{n+1}^{n+1}}$.
\end{proof}

\section{Specht stability}
\label{section:spechtstability}

In this section, we define a different notion of stability for coherent sequences of
representations which we call Specht stability.  There are two key results about
Specht stability.  The first is Theorem \ref{maintheorem:centraltospecht} from \S \ref{section:introduction},
which says that every centrally stable sequence of representations of the symmetric
group is also Specht stable (subject to an assumption on $\Char(\Field)$).  Theorem
\ref{maintheorem:centraltospecht} will be proven in \S \ref{section:centraltospecht}.  The
other key result about Specht stability is the following proposition, which is the analogue of
Proposition \ref{proposition:exactness} for Specht stability.

\begin{proposition}
\label{proposition:spechtexactness}
Let
$$V_1 \rightarrow V_2 \rightarrow \cdots$$
be a coherent sequence of representations of the symmetric group which is Specht stable starting
at $N$.  Consider $n$ and $m$ and $M$ such that $N \leq n \leq m \leq M$.  Then the $M$-central stability chain
complex associated to the potentially centrally stable sequence
$$V_n \rightarrow V_{n+1} \rightarrow \cdots \rightarrow V_m$$
is exact.
\end{proposition}

\noindent
Proposition \ref{proposition:exactness} is an immediate corollary of Proposition \ref{proposition:spechtexactness} and
Theorem \ref{maintheorem:centraltospecht}.

The definition of Specht stability uses the fine structure of the representation theory 
of the symmetric group, which is briefly recalled in \S \ref{section:symmetricgroup}.  
In \S \ref{section:spechtfiltration}, we introduce a special filtration on representations
of $S_n$, and in \S \ref{section:spechtdefinition} we define Specht stability.
Finally, in \S \ref{section:spechtexactnessreduction} we prove Proposition \ref{proposition:spechtexactness}, making
use of a special case of Proposition \ref{proposition:spechtexactness} which is proven in \S \ref{section:spechtexactness}.

\subsection{Review of the representation theory of the symmetric group}
\label{section:symmetricgroup}

We begin by quickly reviewing some background material on the representation theory 
of the symmetric group.  There are numerous very different approaches to this material.
We will follow the approach of James's book \cite{JamesSymmetric}, which is the one that seems
best suited to working in finite characteristic.  Fix a field $\Field$.

\paragraph{Partitions and Young diagrams.}
A {\em partition} $\mu$ of an integer $n$ is an ordered nonincreasing sequence $(\mu_1,\ldots,\mu_k)$
of positive integers whose sum is $n$.  We will often write $\mu \vdash n$ to indicate that
$\mu$ is a partition of $n$.  A partition $\mu = (\mu_1,\ldots,\mu_k)$ can be visualized as a {\em Young diagram},
which is a diagram containing $\mu_1$ empty boxes on the first row, $\mu_2$ on the second row, etc., with all
rows left-justified.  For example,
the Young diagram for $(4,2,1)$ is
$$\ytableausetup{smalltableaux}
\ydiagram{4,2,1}$$
We will frequently confuse a partition with its associated Young diagram; for instance, we will discuss
``adding a box to the upper right hand corner'' of a partition.

\paragraph{Tableaux and tabloids.}
A {\em tableau} of shape $\mu \vdash n$ is obtained by filling in the boxes of the Young diagram of
$\mu$ with the numbers $\{1,\ldots,n\}$ such that each number is used exactly once.  A {\em tabloid}
of shape $\mu$ is similar to a tableau, but the entries in each row are unordered.  If $t$ is a tableau,
then we will denote the tabloid obtained by forgetting the ordering on the rows of $t$ by $\{t\}$.
Let $M^{\mu}(\Field)$ be the set of $\Field$-linear combinations of tabloids of shape $\mu$.  The
group $S_n$ acts on $M^{\mu}(\Field)$ in the obvious way.  It is not hard to see
that $M^{\mu}(\Field) = \Ind_{S_{\mu_1} \times S_{\mu_2} \times \cdots \times S_{\mu_k}}^{S_n} \Field$, where
$S_{\mu_1} \times \cdots \times S_{\mu_k}$ acts trivially on $\Field$ and is embedded in $S_n$ in the obvious way.

\paragraph{Polytabloids and Specht modules.}
The representations $M^{\mu}(\Field)$ are rarely irreducible.  If $t$ is a tableau of shape $\mu \vdash n$, then
let $\ColStab(t)$ be the subgroup of $S_n$ that preserves the columns of $t$.  The {\em polytabloid} $e_t$ associated
to $t$ is then
\begin{equation}
\label{eqn:polytabloid}
e_t = \sum_{\sigma \in \ColStab(t)} (-1)^{|\sigma|} \{\sigma \cdot t\} \in M^{\mu}(\Field).
\end{equation}
The {\em Specht module} associated to $\mu$, denoted $S^{\mu}(\Field)$, is the span of $\Set{$e_t$}{$t$ tableau of shape $\mu$}$ in $M^{\mu}(\Field)$.
The group $S_n$ clearly acts on $S^{\mu}(\Field)$.  A {\em standard tableau}
is a tableau $t$ such that the rows
and columns of $t$ are strictly increasing, and a {\em standard polytabloid} is the polytabloid associated to a standard
tableau.  The set of standard polytabloids of shape $\mu$ forms a basis for $S^{\mu}(\Field)$.

\paragraph{Decomposing representations.}
If $\Char(\Field) = 0$ or $\Char(\Field) \geq n+1$ (in other words, if $\Char(\Field)$ does not divide $n! = |S_n|$), then
$S^{\mu}(\Field)$ is an irreducible $S_n$-representation, and all irreducible $S_n$-representations over $\Field$
arise in this way.  Moreover, the above assumption on $\Char(\Field)$ implies that all representations of $S_n$ are
completely reducible, so we can decompose an $S_n$-representation $V$ over $\Field$ as
$$V = \bigoplus_{i \in I} S^{\mu_i}(\Field),$$
where $\mu_i \vdash n$ for all $i \in I$.  The isotypic components of this decomposition (that is, the direct sums
of isomorphic Specht modules within it) are unique.  We emphasize that all of this holds for infinite-dimensional
$V$, the key point being that if $V$ is an arbitrary $S_n$-representation and $\vec{v} \in V$, then the span
of the orbit $S_n \cdot \vec{v}$ is finite-dimensional.  If $0 < \Char(\Field) \leq n$, then Specht
modules need not be irreducible and $S_n$-representations over $\Field$ need not decompose as direct
sums of irreducible representations.  Nonetheless, the Specht modules still play a basic role in $S_n$-representation
theory.

\paragraph{Restricting representations.}
Fix $\mu \vdash n+k$.  We wish to study $\Res^{S_{n+k}}_{S_n} S^{\mu}(\Field)$.
The {\em deletable rows} of $\mu$ are the rows from which
the right-most box can be deleted to yield a Young diagram (these are the rows that
end with a ``corner'').  A length $k$ {\em deletion sequence} for $\mu$ is an ordered sequence
$\mathfrak{s} = (s_1,\ldots,s_k)$ of rows of $\mu$ such that $s_1$ is a deletable
row of $\mu$, such that $s_2$ is a deletable row of the Young diagram obtained by deleting
the last box in row $s_1$ of $\mu$, etc.  For example, if
$$\mu = \ydiagram{5,2,1}$$
then $(1,2,1)$ is a deletion sequence but $(2,2,1)$ is not a deletion sequence.  
Let $\mu_{\mathfrak{s}}$ denote the Young
diagram obtained by performing this sequence of deletions.  Thus in the previous example, we would
have
$$\mu_{(1,2,1)} = \ydiagram{3,1,1}$$
Let $\mathfrak{S}$ be the set of length $k$ deletion sequences for $\mu$.  We 
then have the following classical {\em restriction rule}.

\begin{theorem}[{\cite[\S 9]{JamesSymmetric}}]
\label{theorem:restriction}
If $\Char(\Field)=0$ or $\Char(\Field) \geq n+1$, 
then $\Res^{S_{n+k}}_{S_n} S^{\mu}(\Field) \cong \bigoplus_{\mathfrak{s} \in \mathfrak{S}} S^{\mu_{\mathfrak{s}}}(\Field)$.
\end{theorem}

\subsection{A filtration on representations of the symmetric group}
\label{section:spechtfiltration}

We now discuss a type of filtration on a representation of the symmetric group which will play
a key role in the rest of this paper.  First, some definitions concerning filtrations.

\begin{definition}
\mbox{}
\begin{itemize}
\item A {\em filtered vector space} of length $N$ is a vector space $V$ equipped with a descending
filtration
$$V = \Filter{N}{V} \supseteq \Filter{N-1}{V} \supseteq \cdots \supseteq \Filter{0}{V} = 0.$$
We will use the convention that $\Filter{i}{V} = V$ for $i > N$ and $\Filter{i}{V} = 0$ for $i \leq 0$.
\item If $f : V \rightarrow W$ is a linear map between filtered vector spaces, then
$f$ is a {\em filtered map} of degree $k \geq 0$ if $f(\Filter{i}{V}) \subset \Filter{i+k}{W}$
for all $i$.  
\item If $f : V \rightarrow W$ is a filtered map of degree $k$, then
we get induced maps
$$f_i : \Filter{i}{V} / \Filter{i+1}{V} \rightarrow \Filter{i+k}{W} / \Filter{i+k+1}{W}$$
for all $i$.  We will call $f_i$ the {\em $i^{\text{th}}$ graded map} associated to $f$.
\end{itemize}
\end{definition}

We now give some motivation for our filtration.  Assume for the moment
that $\Char(\Field)=0$, and let $V$ and $W$ be representations over $\Field$ of $S_n$ and $S_{n+1}$, respectively.
Consider an $S_n$-equivariant map $f : V \rightarrow W$.  Decompose $V$ and $W$ as direct sums
$$V = \bigoplus_{i \in I} S^{\mu_i}(\Field) \quad \text{and} \quad W = \bigoplus_{j \in J} S^{\nu_j}(\Field)$$
of Specht modules.  What can we say about $f(S^{\mu_i}(\Field)) \subset W$?  

A hint is
provided by Theorem \ref{theorem:restriction} (the restriction rule).
Let the first row of $\mu_i$ have $r$ boxes, and let 
$$J' = \Set{$j \in J$}{$\nu_j$ has $r$ or $r+1$ boxes in its first row}.$$
Theorem \ref{theorem:restriction} implies that 
$$f(S^{\mu_i}(\Field)) \subset \bigoplus_{j \in J'} S^{\nu_j}(\Field) \subset W.$$
This suggests that it might be worthwhile to filter a representation of 
$S_n$ by the ``length of the top rows of its Specht modules''.

We now return to considering general fields $\Field$.  We make the above type of filtration precise as follows.

\begin{definition}
Let $V$ be an $S_n$-representation over $\Field$.  
A {\em top-indexed Specht filtration} for $V$ is an $S_n$-invariant filtration
$$V = \Filter{n}{V} \supset \Filter{n-1}{V} \supset \cdots \supset \Filter{0}{V} = 0$$
together with a decomposition
$$\Filter{i}{V} / \Filter{i-1}{V} = \bigoplus_{j \in I_i} S^{\mu(i,j)}(\Field)$$
for each $i$ such that the first row of $\mu(i,j)$ has $i$ boxes for $j \in I_i$.
\end{definition}

\begin{definition}
Let $V$ and $W$ be representations of $S_n$ and $S_{n+1}$, respectively, which are equipped with top-indexed Specht
filtrations.  A {\em Specht filtration map} $f : V \rightarrow W$ is an $S_{n}$-equivariant filtered map
of degree $1$.
\end{definition}

\begin{remark}
It follows from what we said above that if $\Char(\Field)=0$, then all representations $V$ of $S_n$ over $\Field$ can be uniquely equipped with top-indexed Specht
filtrations, and if $W$ is a representation of $S_{n+1}$ over $\Field$ and $f : V \rightarrow W$ is $S_n$-equivariant, then $f$ is a Specht filtration
map.  Neither of these need to hold if $\Char(\Field) > 0$.
\end{remark}

\subsection{Definition of Specht stability}
\label{section:spechtdefinition}

In this section, we define Specht stability.  We begin by describing how to stabilize
a single Specht module.  This notion of stability was first introduced by Church
and Farb; see their paper \cite{ChurchFarbStability} for many examples of situations
``in nature'' in which it occurs.

\begin{definition}
If $\mu = (\mu_1,\ldots,\mu_{\ell}) \vdash n$, then $\stab(\mu) = (\mu_1+1,\mu_2,\ldots,\mu_k) \vdash n+1$.
There is an $S_n$-equivariant map $M^{\mu}(\Field) \hookrightarrow M^{\stab(\mu)}(\Field)$ which
appends an $n+1$ to the first row of a tabloid in $M^{\mu}(\Field)$.  Restricting this
to $S^{\mu}(\Field)$, we get an $S_n$-equivariant map $S^{\mu}(\Field) \hookrightarrow S^{\stab(\mu)}(\Field)$ 
that we will call the {\em stabilization map}.
\end{definition}

\noindent
We now extend this to representations equipped with top-indexed Specht filtrations.

\begin{definition}
Let $V$ and $W$ be representations over $\Field$ of $S_n$ and $S_{n+1}$, respectively.  Assume that $V$ and
$W$ are equipped with top-indexed Specht filtrations and that $f : V \rightarrow W$ is a Specht filtration map.
The map $f$ is a {\em stabilization map} if the following holds for all $i \in \Z$.  Let 
$f_i : \Filter{i}{V} / \Filter{i-1}{V} \rightarrow \Filter{i+1}{W} / \Filter{i}{W}$ be the graded map and let
$$\Filter{i}{V} / \Filter{i-1}{V} = \bigoplus_{j \in I_i} S^{\mu(i,j)}(\Field) \quad \text{and} \quad \Filter{i+1}{W} / \Filter{i}{W} = \bigoplus_{j \in I_i'} S^{\nu(i,j)}(\Field)$$
be the decompositions.  There then exists a bijection $\sigma : I_i \rightarrow I_i'$ such that $f_i$ restricts to the stabilization
map $S^{\mu(i,j)}(\Field) \rightarrow S^{\nu(i,\sigma(i))}(\Field)$ for all $j \in I_i$.
\end{definition}

\begin{remark}
If $f : V \rightarrow W$ is a stabilization map as in the previous definition, then since $\Filter{0}{V}=0$ we must
have $\Filter{1}{W} = 0$.
\end{remark}

\noindent
We can now define Specht stability.

\begin{definition}
Let
$$V_1 \longrightarrow V_2 \longrightarrow V_3 \longrightarrow V_4 \longrightarrow \cdots$$
be a coherent sequence of representations of the symmetric group.  This sequence
is {\em Specht stable} with stability starting at $N$ if for all $n \geq N$, the $S_n$-representation
$V_n$ can be equipped with a top-indexed Specht filtration $\Filter{\bullet}{V_n}$
such that the maps $V_n \rightarrow V_{n+1}$ are stabilization maps.
\end{definition}

\begin{remark}
If 
$$V_1 \longrightarrow V_2 \longrightarrow V_3 \longrightarrow V_4 \longrightarrow \cdots$$
is a coherent sequence of representations of the symmetric group which is Specht stable starting at $N$, then
by a reasoning similar to the remark after the definition of the stabilization map we must have
$\Filter{i}({V_n})=0$ for $n \geq N$ and $i \leq n-N$.
\end{remark}

\begin{remark}
For coherent sequences of finite-dimensional representations over a field of characteristic $0$,
Specht stability is easily seen to imply both representation stability in the sense of 
Church-Farb \cite{ChurchFarbStability} and monotonicity in the sense of Church \cite{ChurchConfiguration}.
\end{remark}

\subsection{Reduction of Proposition \ref{proposition:spechtexactness} to a special case}
\label{section:spechtexactnessreduction}

The following is a special case of Proposition \ref{proposition:spechtexactness}.

\begin{proposition}
\label{proposition:spechtresolution}
Fix $\mu \vdash n$ and $k \geq 1$.  Let
\begin{equation}
\label{eqn:spechtchain}
\Induce{S^{\mu}(\Field)}{k} \rightarrow \Induce{S^{\stab(\mu)}(\Field)}{k-1} \rightarrow \cdots \rightarrow S^{\stab^k(\mu)}(\Field) \rightarrow 0
\end{equation}
be the $(n+k)$-central stability chain complex associated to the potentially stable sequence
$$S^{\mu}(\Field) \rightarrow S^{\stab(\mu)}(\Field) \rightarrow \cdots \rightarrow S^{\stab^k(\mu)}(\Field)$$
of representations of the symmetric group.  Then \eqref{eqn:spechtchain} is exact.
\end{proposition}

\noindent
The proof of Proposition \ref{proposition:spechtresolution} is in \S \ref{section:spechtexactness}.  

The following corollary follows from Lemma \ref{lemma:centralstab}
and the case $k=2$ of Proposition \ref{proposition:spechtresolution}.

\begin{corollary}
\label{corollary:stabilizespecht}
For $\mu \vdash n$, we have
$\Stab{S^{\mu}(\Field)}{S^{\stab(\mu)}(\Field)} \cong S^{\stab^2(\mu)}(\Field)$.  Moreover, the
map $S^{\stab(\mu)}(\Field) \rightarrow S^{\stab^2(\mu)}(\Field)$ obtained by composing the
map $S^{\stab(\mu)}(\Field) \hookrightarrow \Induce{S^{\stab(\mu)}(\Field)}{1}$ with the projection
$\Induce{S^{\stab(\mu)}(\Field)}{1} \rightarrow \Stab{S^{\mu}(\Field)}{S^{\stab(\mu)}(\Field)}$ is
the stabilization map.
\end{corollary}

We now show how to derive Proposition \ref{proposition:spechtexactness} from Proposition \ref{proposition:spechtresolution}.

\begin{proof}[{Proof of Proposition \ref{proposition:spechtexactness}}]
Let us recall the setup.  Let
$$V_1 \rightarrow V_2 \rightarrow \cdots$$
be a coherent sequence of representations of the symmetric group which is Specht stable starting
at $N$.  Consider $n$ and $m$ and $M$ such that $N \leq n \leq m \leq M$.  Then the claim is that
the $M$-central stability chain complex 
\begin{equation}
\label{eqn:spechtchain2}
\Induce{V_n}{M-n} \rightarrow \Induce{V_{n+1}}{M-n-1} \rightarrow \cdots \rightarrow \Induce{V_m}{M-m}
\end{equation}
associated to the potentially centrally stable sequence
$$V_n \rightarrow V_{n+1} \rightarrow \cdots \rightarrow V_m$$
is exact.  The filtrations on the $V_i$ given by Specht stability induce filtrations on the terms
$\Induce{V_i}{M-i}$.  These are compatible with the differentials in \eqref{eqn:spechtchain2}, so
\eqref{eqn:spechtchain2} is a filtered chain complex.  The associated graded pieces of
this filtered chain complex are direct sums of chain complexes like in Proposition \ref{proposition:spechtresolution},
and thus by Proposition \ref{proposition:spechtresolution} the homology of the associated graded
pieces of \eqref{eqn:spechtchain2} vanish.  Standard homological algebra (for instance, the
spectral sequence of a filtered chain complex) then shows that the homology of \eqref{eqn:spechtchain2} vanishes.
\end{proof}

\subsection{Specht stability and the central stability chain complex}
\label{section:spechtexactness}

Our goal is to prove Proposition \ref{proposition:spechtresolution}.  It turns out that
this was essentially proven by G. James in \cite[\S 17]{JamesSymmetric}, though his
formulation is quite different and it takes some effort to extract Proposition \ref{proposition:spechtresolution}
from James's work.  We begin by going over some necessary representation-theoretic background material.

\paragraph{Duality.}
If $\nu \vdash n$, then the {\em conjugate partition} of $\nu$, denoted $\nu'$, is the partition
whose Young diagram is obtained by converting each row of $\nu$ into a column.  For instance,
$$\nu = \ytableausetup{smalltableaux}\ydiagram{4,2,1} \quad \quad \text{and} \quad \quad \nu' =
\ydiagram{3,2,1,1}$$
Recalling that $\Alternate{n}$ is the sign representation of $S_n$, 
define $S^{\nu}_{\Alternate{}}(\Field) = S^{\nu}(\Field) \otimes \Alternate{n}$.  It is
then classical that $(S^{\nu}_{\Alternate{}}(\Field))^{\ast} \cong S^{\nu'}(\Field)$ (see, e.g.,\ \cite[\S 4]{Fayers}).  
For a representation $V$ of $S_n$ and $j \geq 0$, define
$$\InduceTriv{V}{j} = \text{Ind}_{S_n \times S_j}^{S_{n+j}} V \boxtimes \Trivial{j},$$
where $\Trivial{k}$ is the trivial representation of $S_j$.
Recall that if $H$ is a subgroup of $G$ and $W$ is an $H$-representation, then $(\Ind_H^G W)^{\ast} \cong \Ind_H^G W^{\ast}$
(see, e.g., \cite[\S 3.3]{BensonRep}).  This implies that $(\InduceTriv{S^{\nu}_{\Alternate{}}(\Field)}{j})^{\ast} \cong \InduceTriv{S^{\nu'}(\Field)}{j}$.

\paragraph{Weak partitions.}
A {\em weak partition} $\eta$ of an integer $n$ is an ordered sequence $(\eta_1,\ldots,\eta_k)$
of nonnegative integers whose sum is $n$.  Young diagrams, tableau and tabloids of shape $\eta$ are defined
in the obvious way, and given a tableau $t$ of shape $\eta$, we will let $\{t\}$ denote
the associated tabloid.  Also, $M^{\eta}(\Field)$ will still denote
the set of $\Field$-linear combinations of tabloids of shape $\eta$.  If $\eta=(\eta_1,\ldots,\eta_k)$
and $\nu=(\nu_1,\ldots,\nu_{\ell})$ are weak partitions with $\ell \leq k$, then
we will write $\nu \subset \eta$ if $\nu_i \leq \eta_i$ for all $1 \leq i \leq k$, where by convention
$\nu_i=0$ for $\ell < i \leq k$.  If $\nu \subset \eta$, then we will regard the Young diagram
of $\nu$ as being contained in the Young diagram for $\eta$.  For example, if $\nu = (2,1)$ and
$\eta=(3,1,3)$, then the Young diagrams are as follows.
$$\ydiagram[*(white) \bullet] {2,1}*[*(white)]{3,1,3}$$
The $\bullet$'s indicate the location of $\nu$.  Given a tableau $t$ of shape $\eta$, this allows us to refer to the portion of $t$ lying inside/outside $\nu$.

\paragraph{Specht modules for weak partitions.}
Assume now that $\eta$ is a weak partition of $n$ and $\nu$ is a partition (not just a weak
partition) satisfying $\nu \subset \eta$.  If $t$ is a tableau of shape $\eta$, then
let $\ColStab(t,\nu)$ be the subgroup of $S_n$ that acts as the identity on the portion
of $t$ lying outside of $\nu$ and preserves the columns of the portion of $t$ lying inside $\nu$.  
The {\em polytabloid} $e_t^{\nu}$ associated to $t$ is then
\begin{equation}
\label{eqn:genpolytabloid}
e_t^{\nu} = \sum_{\sigma \in \ColStab(t,\nu)} (-1)^{|\sigma|} \{\sigma \cdot t\} \in M^{\eta}(\Field).
\end{equation}
The {\em generalized Specht module} associated to the pair $(\nu,\eta)$, denoted $S^{\nu,\eta}(\Field)$, is the span in $M^{\eta}(\Field)$ of the set
$\Set{$e_t^{\nu}$}{$t$ tableau of shape $\eta$}$.  The group $S_n$ clearly acts on $S^{\nu,\eta}(\Field)$.  

\paragraph{Adding tails and stabilizing.}
Generalized Specht modules are closely related to certain kinds of induced representations.
For this, we need some notation.  Consider a partition $\nu = (\nu_1,\ldots,\nu_{\ell})$.
For $k \geq 1$, define 
$$\nu[k] = (\nu_1,\ldots,\nu_{\ell}+k) \quad \quad \text{and} \quad \quad \hatstab^k(\nu) = (\nu_1,\ldots,\nu_{\ell},k).$$
We will omit the $k$ in $\hatstab^k(\nu)$ if $k=1$.  Also, we will use the conventions $\hatstab^0(\nu) = \nu$ and $\nu[0] = \nu$.
We then have the following.

\begin{theorem}[{\cite[Theorem 17.13]{JamesSymmetric}}]
\label{theorem:decomposeinduce}
Let $\nu$ be a partition of $n$.  The for $m > n$, there exists a short exact sequence
$$0 \longrightarrow S^{\hatstab(\nu),(\hatstab(\nu))[m-n-1]}(\Field) \longrightarrow \InduceTriv{S^{\nu}(\Field)}{m-n} \longrightarrow S^{\nu,\nu[m-n]}(\Field) \longrightarrow 0$$
of $S_m$-representations.
\end{theorem}

\begin{remark}
To relate Theorem \ref{theorem:decomposeinduce} to the statement in \cite[Theorem 17.13]{JamesSymmetric}, we make the following
two remarks.
\begin{itemize}
\item In the notation of \cite[Theorem 17.13]{JamesSymmetric}, we are taking $\mu^{\#} = \nu$ and $\mu = (\nu_1,\ldots,\nu_{\ell},m-n)$.
\item Instead of $\InduceTriv{S^{\nu})(\Field)}{m-n}$, the statement of \cite[Theorem 17.13]{JamesSymmetric} has $S^{\mu^{\#},\mu}(\Field)$ (though
actually, the field $\Field$ is not specified in the notation in \cite{JamesSymmetric}).  The isomorphism 
$S^{\mu^{\#},\mu}(\Field) \cong \InduceTriv{S^{\nu})(\Field)}{m-n}$ is discussed in the proof of \cite[Corollary 17.14]{JamesSymmetric}.
\end{itemize}
\end{remark}

\paragraph{An exact sequence.}
Let $\nu$ be a partition of $n$.  Consider $k \geq 0$.  Theorem \ref{theorem:decomposeinduce} gives the following
short exact sequences.
\begin{align*}
0 \longrightarrow S^{\hatstab(\nu),(\hatstab(\nu))[k-1]}(\Field)     \longrightarrow &\InduceTriv{S^{\nu}(\Field)}{k}                 \longrightarrow S^{\nu,\nu[k]}(\Field) \longrightarrow 0 \\
0 \longrightarrow S^{\hatstab^2(\nu),(\hatstab^2(\nu))[k-2]}(\Field) \longrightarrow &\InduceTriv{S^{\hatstab(\nu)}(\Field)}{k-1}       \longrightarrow S^{\hatstab(\nu),(\hatstab(\nu))[k-1]}(\Field) \longrightarrow 0 \\
&\vdots \\
0 \longrightarrow S^{\hatstab^{k-1}(\nu),(\hatstab^{k-1}(\nu))[1]}(\Field)   \longrightarrow &\InduceTriv{S^{\hatstab^{k-2}(\nu)}(\Field)}{2} \longrightarrow S^{\hatstab^{k-2}(\nu),(\hatstab^{k-2}(\nu))[2]}(\Field) \longrightarrow 0 \\
0 \longrightarrow S^{\hatstab^k(\nu),(\hatstab^k(\nu))[0]}(\Field)   \longrightarrow &\InduceTriv{S^{\hatstab^{k-1}(\nu)}(\Field)}{1} \longrightarrow S^{\hatstab^{k-1}(\nu),(\hatstab^{k-1}(\nu))[1]}(\Field) \longrightarrow 0 
\end{align*}
Stringing these short exact sequences together and using the obvious isomorphism $S^{\hatstab^k(\nu),(\hatstab^k(\nu))[0]}(\Field) \cong S^{\hatstab^k(\nu)}(\Field)$,
we obtain the following.

\begin{corollary}
\label{corollary:hatstabseq}
Let $\nu$ be a partition of $n$ and let $k \geq 0$.  There is then an exact sequence
$$0 \longrightarrow S^{\hatstab^k(\nu)}(\Field) \longrightarrow \InduceTriv{S^{\hatstab^{k-1}(\nu)}(\Field)}{1} \longrightarrow \cdots \longrightarrow \InduceTriv{S^{\nu}(\Field)}{k}$$
of representations of $S_{n+k}$.
\end{corollary}

\paragraph{The proof.}
We are finally in a position to prove Proposition \ref{proposition:spechtresolution}.

\begin{proof}[{Proof of Proposition \ref{proposition:spechtresolution}}]
Let us recall the setup.  Fix $\mu \vdash n$ and $k \geq 1$.  Let
\begin{equation}
\label{eqn:spechtchain3}
\Induce{S^{\mu}(\Field)}{k} \longrightarrow \Induce{S^{\stab(\mu)}(\Field)}{k-1} \longrightarrow \cdots \longrightarrow S^{\stab^k(\mu)}(\Field) \longrightarrow 0
\end{equation}
be the $(n+k)$-central stability chain complex associated to the potentially stable sequence
$$S^{\mu}(\Field) \longrightarrow S^{\stab(\mu)}(\Field) \longrightarrow \cdots \longrightarrow S^{\stab^k(\mu)}(\Field)$$
of representations of the symmetric group.  Then we must prove that \eqref{eqn:spechtchain3} is exact.

Recall that if $H$ is a subgroup of $G$ and $V$ is an $H$-representation and $W$ is a $G$-representation, then
$W \otimes \Ind_H^G V \cong \Ind_H^G (V \otimes \Res^G_H W)$ (see, e.g., \cite[Proposition 3.3.3i]{BensonRep}).  
This implies that $\Alternate{n+k} \otimes \Induce{S^{\stab^j(\mu)}(\Field)}{k-j} \cong \InduceTriv{S^{\stab^j(\mu)}_{\Alternate{}}(\Field)}{k-j}$ for
all $0 \leq j \leq k$.  Tensoring \eqref{eqn:spechtchain3} with $\Alternate{n+k}$, it is therefore enough to prove
that the resulting chain complex
\begin{equation}
\label{eqn:spechtchain4}
\InduceTriv{S^{\mu}_{\Alternate{}}(\Field)}{k} \longrightarrow \InduceTriv{S^{\stab(\mu)}_{\Alternate{}}(\Field)}{k-1} \longrightarrow \cdots \longrightarrow S^{\stab^k(\mu)}_{\Alternate{}}(\Field) \longrightarrow 0
\end{equation}
is exact.  

Recall from above that $\InduceTriv{S^{\stab^{k-j}(\mu)}_{\Alternate{}}(\Field)}{j}$ is dual to $\InduceTriv{S^{(\stab^{k-j}(\mu))'}(\Field)}{j}$ for all
$0 \leq j \leq k$.  Since $(\stab^{k-j}(\mu))' = \hatstab^{k-j}(\mu')$, the dual chain complex of \eqref{eqn:spechtchain4} is
\begin{equation}
\label{eqn:spechtchain5}
0 \longrightarrow S^{\hatstab^k(\mu')}(\Field) \longrightarrow \InduceTriv{S^{\hatstab^{k-1}(\mu')}(\Field)}{1} \longrightarrow \cdots \longrightarrow \InduceTriv{S^{\mu'}(\Field)}{k}.
\end{equation}
It is enough to prove that \eqref{eqn:spechtchain5} is exact.  In fact, letting $\nu = \mu'$, this is exactly
the chain complex that Corollary \ref{corollary:hatstabseq} asserts is exact (modulo the signs of the boundary
maps, which depend on the noncanonical choice of a duality pairing).  This follows easily from
the formulas for the various maps involved given in \cite[\S 17]{JamesSymmetric} combined with the explicit
duality isomorphism given in \cite[\S 4]{Fayers}.
\end{proof}

\section{Central stability implies Specht stability}
\label{section:centraltospecht}

In this section, we will prove Theorem \ref{maintheorem:centraltospecht}, which asserts
that a centrally stable sequence of representations of the symmetric group
is also Specht stable.  We start with
some definitions.  

\begin{definition}
The {\em width} of a Specht module $S^{\mu}(\Field)$ is the number of boxes in the first row of
$\mu$.
\end{definition}

\begin{definition}
Let $\phi_{N-1} : V_{N-1} \rightarrow V_N$ be an $S_{N-1}$-equivariant map from a representation
of $S_{N-1}$ to a representation of $S_N$ and let $Q_{N+1}$ be an $S_{N+1}$-subrepresentation
of $\Stabb{V_{N-1}}{V_{N}}{\phi_{N-1}}$.  The {\em quotiented central stabilization sequence} associated
to $\phi_{N-1}$ and $Q_{N+1}$ is the sequence
$$V_{N-1} \stackrel{\phi_{N-1}}{\longrightarrow} V_N \stackrel{\phi_N}{\longrightarrow} V_{N+1} \stackrel{\phi_{N+1}}{\longrightarrow} V_{N+2} \stackrel{\phi_{N+2}}{\longrightarrow} \cdots$$
which is inductively defined as follows.  First, $V_{N+1} = \Stabb{V_{N-1}}{V_{N}}{\phi_{N-1}} / Q_{N+1}$ and
$\phi_N$ is the natural map.  Next, assume that $V_{n-1}$ and $V_n$ and $\phi_{n-1} : V_{n-1} \rightarrow V_n$
are defined for some $n \geq N+1$.  Then $V_{n+1} = \Stabb{V_{n-1}}{V_n}{\phi_{n-1}}$ and
$\phi_n : V_n \rightarrow V_{n+1}$ is the natural map.  If $Q_{N+1}=0$, then we will simply call
this the {\em central stabilization sequence} associated to $\phi_{N-1}$.
\end{definition}

\noindent
Our first lemma restricts the Specht modules that can appear in a central stabilization sequence.
Recall that if either $\Char(\Field)=0$ or $\Char(\Field) \geq n+1$, then
every representation of $S_n$ can be decomposed
into a direct sum of Specht modules and the isotypic components of this decomposition are unique.

\begin{lemma}
\label{lemma:widthlowerbound}
Let $\phi_{N-1} : V_{N-1} \rightarrow V_N$ be an $S_{N-1}$-equivariant map from a representation
of $S_{N-1}$ to a representation of $S_N$ and let
$$V_{N-1} \stackrel{\phi_{N-1}}{\longrightarrow} V_N \stackrel{\phi_N}{\longrightarrow} V_{N+1} \stackrel{\phi_{N+1}}{\longrightarrow} V_{N+2} \stackrel{\phi_{N+2}}{\longrightarrow} \cdots$$
be the associated central stabilization sequence.  Consider $n \geq N$.  Assume that either
$\Char(\Field)=0$ or $\Char(\Field) \geq n+1$.  Then every Specht module that occurs in $V_n$ has width
at least $n-N$.
\end{lemma}
\begin{proof}
Let $S^{\mu}(\Field)$ be a Specht module that occurs in $V_n$.
Recall that $\Trivial{k} \cong \Field$ is the trivial representation of $S_k$.  By construction,
$V_n$ is a quotient of $V_n':=\Ind_{S_N \times S_{n-N}}^{S_n} V_N \boxtimes \Trivial{n-N}$, so
$S^{\mu}(\Field)$ appears in $V_n'$.  The Littlewood-Richardson rule 
says that there exists some $\nu \vdash N$ such that $S^{\nu}(\Field)$ appears in
$V_N$ and such that $\mu$ is obtained by adding $n-N$ boxes to $\nu$ with no two boxes
added to the same column (see (\cite[\S 16]{JamesSymmetric}; the special case we are using is often called 
Pieri's formula).  Thus $\mu$ has at least $n-N$ columns, and hence
at least $n-N$ boxes in its first row.
\end{proof}

We now need three more definitions.

\begin{definition}
Let $V_n$ be a representation of $S_n$.  Assume that either $\Char(\Field) = 0$ or $\Char(\Field) \geq n+1$.
Then the {\em width} of $V_n$ is the maximum width of a Specht module that occurs in $V_n$.
\end{definition}

\noindent
In the following definition, observe that there is no assumption on $\Char(\Field)$.

\begin{definition}
Let $V_n$ be a representation of $S_n$.  We say that $V_n$ has {\em constant width $k$} if it can
be decomposed into a direct sum of Specht modules of width $k$.  We say that an $S_n$-subrepresentation
$W_n$ of $V_n$ has {\em constant cowidth $k$} if $V_n/W_n$ has constant width $k$.
\end{definition}

\begin{remark}
A theorem of Hemmer-Nakano \cite{HemmerNakano} says that if a representation of $S_n$ can be decomposed
into a direct sum of Specht modules, then the Specht modules that occur are independent
of the decomposition (as long as $\Char(\Field) \geq 5$).
\end{remark}

\noindent
Recall that if $\mu \vdash n$, then there
is a natural stabilization map $S^{\mu}(\Field) \hookrightarrow S^{\stab(\mu)}(\Field)$.

\begin{definition}
Let $\phi_{n} : V_{n} \rightarrow V_{n+1}$ be an $S_{n}$-equivariant map from a representation 
of $S_{n}$ to a representation of $S_{n+1}$.  Assume that $V_n$ has constant width $k$, and let
$$V_n = \bigoplus_{i \in I} S^{\mu_i}(\Field)$$
be the associated decomposition.  Then $\phi_n$ is a {\em stabilization map} if we can write
$$V_{n+1} = \bigoplus_{i \in I} S^{\stab(\mu_i)}(\Field)$$
such that the restriction of $\phi_n$ to $S^{\mu_i}(\Field)$ is the stabilization map
$S^{\mu_i}(\Field) \hookrightarrow S^{\stab(\mu_i)}(\Field)$.
\end{definition}

\begin{lemma}
\label{lemma:weakstabilization}
Let $V_n$ be a representation of $S_n$, let $V_{n+1}$ be an $S_{n+1}$ representation obtained
as a quotient of $\Induce{V_n}{1}$, and let $\phi_n : V_n \rightarrow V_{n+1}$ be the natural
$S_{n}$-equivariant map.  Assume that either $\Char(\Field)=0$ or
$\Char(\Field) \geq n+2$.  Let $k$ be the width of $V_n$, 
let $W_n$ be the subspace of $V_n$ spanned by Specht modules of width strictly less
than $k$, and let $W_{n+1}$ be the subspace of $V_{n+1}$ spanned by Specht modules
of width strictly less than $k+1$.  The following then hold.
\begin{itemize}
\item $\phi_n(W_n) \subset W_{n+1}$, so there is an induced map 
$\hat{\phi}_n : V_n/W_n \rightarrow V_{n+1}/W_{n+1}$.
\item We can factor $\hat{\phi}_n$ as 
$V_n/W_n \stackrel{\hat{\phi}_n'}{\rightarrow} \hat{V}_n' \stackrel{\hat{\phi}_n''}{\rightarrow} V_{n+1}/W_{n+1}$,
where $\hat{\phi}_n'$ is a surjection and $\hat{\phi}_n''$ is a stabilization map.
\end{itemize}
\end{lemma}
\begin{proof}
The restriction rule (Theorem \ref{theorem:restriction}) implies that $\phi_n(W_n) \subset W_{n+1}$,
so we concentrate on the second claim.  To simplify our notation, write $\hat{V}_n = V_n/W_n$
and $\hat{V}_{n+1} = V_{n+1} / W_{n+1}$

Assume first that $V_n = \hat{V}_n = S^{\mu}(\Field)$
and $V_{n+1} = \Induce{V_n}{1}$.  The
branching rule \cite[\S 9.2]{JamesSymmetric} implies that $\hat{V}_{n+1} = S^{\stab(\mu)}(\Field)$.
The universal property 
$\Hom_{S_n}(V_n,\Res_{S_n}^{S_{n+1}} \hat{V}_{n+1}) = \Hom_{S_{n+1}}(\Induce{V_n}{1},\hat{V}_{n+1})$ implies
that $\hat{\phi}_{n} \neq 0$.  Since $V_n$ is irreducible, $\hat{\phi}_n$ must be injective.  The restriction
rule (Theorem \ref{theorem:restriction}) then implies that $\hat{\phi}_n$ is the stabilization map.

We now consider general $V_n$.  Let $\hat{V}_{n+1}'$ be the quotient of $\Induce{V_n}{1}$ by the
subspace spanned by all Specht modules of width strictly less than $k+1$.  There then
exists an $S_{n+1}$-subrepresentation $Q_{n+1}$ of $\hat{V}_{n+1}'$ such that 
$\hat{V}_{n+1} = \hat{V}_{n+1}' / Q_{n+1}$.  We can write
$$\hat{V}_{n+1}' = \bigoplus_{j \in J} S^{\nu_j}(\Field) \quad \text{and} \quad Q_{n+1} = \bigoplus_{j \in J'} S^{\nu_j}(\Field)$$
with $J' \subset J$.  It is an easy exercise using the results in the previous paragraph
together with Schur's lemma to show that we can write
$$\hat{V}_n = \bigoplus_{j \in J} S^{\mu_j}(\Field),$$
where for all $j \in J$ we have $\nu_j = \stab(\mu_j)$ and the restriction of the natural map
$\hat{V}_n \rightarrow \hat{V}_{n+1}'$ to $S^{\mu_j}(\Field)$ is the stabilization map
$S^{\mu_j}(\Field) \hookrightarrow S^{\nu_j}(\Field)$.  We can then let $\hat{V}_n' = \hat{V}_n/Q_n$,
where
$$Q_n = \bigoplus_{j \in J'} S^{\mu_j}(\Field) \subset \hat{V}_n,$$
and let $\hat{\phi}_n' : \hat{V}_n \rightarrow \hat{V}_n'$ and 
$\hat{\phi}_n'' : \hat{V}_n' \rightarrow \hat{V}_{n+1}$ be the natural maps.
\end{proof}

\begin{lemma}
\label{lemma:firstrow}
Let $\phi_{N-1} : V_{N-1} \rightarrow V_N$ be an $S_{N-1}$-equivariant map from a representation
of $S_{N-1}$ to a representation of $S_N$ and let $Q_{N+1}$ be an $S_{N+1}$-subrepresentation
of $\Stabb{V_{N-1}}{V_{N}}{\phi_{N-1}}$.  Let
$$V_{N-1} \stackrel{\phi_{N-1}}{\longrightarrow} V_N \stackrel{\phi_N}{\longrightarrow} V_{N+1} \stackrel{\phi_{N+1}}{\longrightarrow} V_{N+2} \stackrel{\phi_{N+2}}{\longrightarrow} \cdots,$$
be the associated quotiented central stabilization sequence.  Assume that either $\Char(\Field)=0$
or $\Char(\Field) \geq N+2$, and let $k$ be the width of $V_N$.
For all $n \geq N$, there then exist constant cowidth $k+(n-N)$ subrepresentations $W_n$ of $V_n$ 
such that the following hold.
\begin{enumerate}
\item The representations $W_N$ and $W_{N+1}$ have width at most $k-1$ and $k$, respectively.
\item For $n \geq N$, we have $\phi_n(W_n) \subset W_{n+1}$.  Moreover, for $n \geq N+1$ the
induced map $\hat{\phi}_n : V_n/W_n \rightarrow V_{n+1}/W_{n+1}$ is a stabilization map.
\item For $n \geq N$, let $\phi_n' : W_n \rightarrow W_{n+1}$ be the restriction of $\phi_n$.
There then exists some $S_{N+2}$-subrepresentation $Q_{N+2}'$ of $\Stabb{W_N}{W_{N+1}}{\phi_N'}$ such that
the sequence 
$$W_{N} \stackrel{\phi_{N}'}{\longrightarrow} W_{N+1} \stackrel{\phi_{N+1}'}{\longrightarrow} W_{N+2} \stackrel{\phi_{N+2}'}{\longrightarrow} W_{N+3} \stackrel{\phi_{N+3}'}{\longrightarrow} \cdots$$
is the quotiented central stabilization sequence associated to $\phi_N'$ and $Q_{N+2}'$.
\end{enumerate}
\end{lemma}
\begin{proof}
Let $W_N$ (resp.\ $W_{N+1}$) be the subspace of $V_{N}$ (resp.\ $V_{N+1}$) spanned by Specht
modules of width strictly less than $k$ (resp.\ $k+1$).  Condition 1 is clearly satisfied, and
Lemma \ref{lemma:weakstabilization} says that $\phi_N(W_N) \subset W_{N+1}$.  Assume now
that $n \geq N+1$ and that we have constructed $W_N,\ldots,W_{n}$ satisfying the conclusions of
the lemma.  

\BeginSteps
\begin{step}
We construct $W_{n+1}$.
\end{step}

\noindent
We need some notation.  
\begin{itemize}
\item Let $\phi_{n-1}' : W_{n-1} \rightarrow W_n$ be the restriction of $\phi_{n-1}$ and let
$\hat{\phi}_{n-1} : V_{n-1}/W_{n-1} \rightarrow V_n/W_n$ be the induced map.
\item Let $\partial_{n-1} : \Induce{V_{n-1}}{2} \rightarrow \Induce{V_n}{1}$ and 
$\partial_{n-1}' : \Induce{W_{n-1}}{2} \rightarrow \Induce{W_n}{1}$ and $\hat{\partial}_{n-1} : \Induce{V_{n-1}/W_{n-1}}{2} \rightarrow \Induce{V_n/W_n}{1}$
be the $(n+1)$-boundary maps associated to $\phi_{n-1}$ and $\phi_{n-1}'$ and $\hat{\phi}_{n-1}$, respectively.
\item Let $\widetilde{W}_{n+1} = \Stabb{W_{n-1}}{W_n}{\phi_{n-1}'}$.
\end{itemize}
Lemma \ref{lemma:centralstab} together with our assumptions implies that $V_{n+1} = \Coker(\partial_{n-1})$ and
$\widetilde{W}_{n+1} = \Coker(\partial_{n-1}')$.  For
all $k \geq 0$, the functor $\Induce{\bullet}{k}$ is exact.  We thus have a commutative diagram
$$
\xymatrix{
0 \ar[r] & \Induce{W_{n-1}}{2} \ar[r] \ar[d]_{\partial_{n-1}'} & \Induce{V_{n-1}}{2} \ar[r] \ar[d]_{\partial_{n-1}} & \Induce{V_{n-1}/W_{n-1}}{2} \ar[r] \ar[d]_{\hat{\partial}_{n-1}} & 0 \\
0 \ar[r] & \Induce{W_n}{1}     \ar[r] \ar@{->>}[d]             & \Induce{V_n}{1}     \ar[r] \ar@{->>}[d]            & \Induce{V_n/W_n}{1}         \ar[r]                               & 0 \\
         & \widetilde{W}_{n+1}                                 & V_{n+1}                                            &                                                                  &}$$
whose rows and columns are exact.
We can form $V_{n+1} = \Coker(\partial_{n-1})$ in two steps.  First, 
let $\widetilde{V}_{n+1} = \Induce{V_n}{1} / \partial_{n-1}(\Induce{W_{n-1}}{2})$.  Chasing the above diagram,
we see that there is an exact sequence
$$0 \longrightarrow \widetilde{W}_{n+1} \longrightarrow \widetilde{V}_{n+1} \longrightarrow \Induce{V_n/W_n}{1} \longrightarrow 0.$$
Also, the map $\Induce{V_{n-1}}{2} \rightarrow \widetilde{V}_{n+1}$ factors through a map
$\overline{\partial}_{n-1} : \Induce{V_{n-1}/W_{n-1}}{2} \rightarrow \widetilde{V}_{n+1}$ satisfying
$V_{n+1} = \Coker(\overline{\partial}_{n-1})$.  Let $W_{n+1}$ be the image of $\widetilde{W}_{n+1}$ in $V_{n+1}$.

\begin{step}
$W_{n+1}$ satisfies the second conclusion of the lemma.
\end{step}

\noindent
We have $\phi_n(W_n) \subset W_{n+1}$ by construction, so
we must show that the induced map $\hat{\phi}_n : V_{n}/W_{n} \rightarrow V_{n+1}/W_{n+1}$ is a stabilization map.  Observe
that we have a commutative diagram
$$\xymatrix{
         &                            & \Induce{V_{n-1}/W_{n-1}}{2} \ar[d]_{\overline{\partial}_{n-1}} \ar[dr]^{\hat{\partial}_{n-1}} &&\\
0 \ar[r] & \widetilde{W}_{n+1} \ar[r] \ar@{->>}[d] & \widetilde{V}_{n+1} \ar[r] \ar@{->>}[d] & \Induce{V_n/W_n}{1} \ar[r] & 0 \\
0 \ar[r] & W_{n+1} \ar[r]                          & V_{n+1} \ar[r]                          & V_{n+1}/W_{n+1} \ar[r]     & 0}$$
whose rows and columns are exact.  Chasing this diagram, we deduce that there is a short exact sequence
\begin{equation}
\label{eqn:vmodwstabilize}
\Induce{V_{n-1}/W_{n-1}}{2} \stackrel{\hat{\partial}_{n-1}}{\longrightarrow} \Induce{V_n/W_n}{1} \longrightarrow V_{n+1}/W_{n+1} \longrightarrow 0.
\end{equation}
Lemma \ref{lemma:centralstab} then implies that $V_{n+1}/W_{n+1} = \Stab{V_{n-1}/W_{n-1}}{V_n/W_n}$.

There are now two cases.  If $n \geq N+2$, then the map $V_{n-1}/W_{n-1} \rightarrow V_n/W_n$ is a stabilization map by induction,
so Corollary \ref{corollary:stabilizespecht} implies that $\hat{\phi}_n$ is a stabilization map.  If instead $n = N+1$, then
Lemma \ref{lemma:weakstabilization} says that we can factor $\hat{\phi}_{n-1}$ as a composition
$$V_{n-1}/W_{n-1} \stackrel{\hat{\phi}_{n-1}'}{\rightarrow} \hat{V}_{n-1}' \stackrel{\hat{\phi}_{n-1}''}{\rightarrow} V_{n}/W_{n},$$
where $\hat{\phi}_{n-1}'$ is a surjection and $\hat{\phi}_{n-1}''$ is a stabilization map.  Letting $\hat{\partial}_{n-1}''$ be
the $(n+1)$-boundary map associated to $\hat{\phi}_{n-1}''$, we can therefore factor $\hat{\partial}_{n-1}$ as a composition
$$\Induce{V_{n-1}/W_{n-1}}{2} \twoheadlongrightarrow \Induce{\hat{V}_{n-1}'}{2} \stackrel{\hat{\partial}_{n-1}''}{\longrightarrow} \Induce{V_n/W_n}{1}.$$
Combining this with \eqref{eqn:vmodwstabilize}, we obtain an exact sequence
$$\Induce{\hat{V}_{n-1}'}{2} \stackrel{\hat{\partial}_{n-1}''}{\longrightarrow} \Induce{V_n/W_n}{1} \longrightarrow V_{n+1}/W_{n+1} \longrightarrow 0,$$
so we can apply Lemma \ref{lemma:centralstab} to deduce that $V_{n+1}/W_{n+1} = \Stab{\hat{V}_{n-1}'}{V_n/W_n}$ and then apply
Corollary \ref{corollary:stabilizespecht} to deduce that $\hat{\phi}_n$ is a stabilization map.

\begin{step}
$W_{n+1}$ satisfies the third conclusion of the lemma.
\end{step}

\noindent
Letting $Q_{n+1}' = \widetilde{W}_{n+1} \cap \Image(\overline{\partial}_{n-1})$, we have a short exact sequence
$$0 \longrightarrow Q_{n+1}' \longrightarrow \widetilde{W}_{n+1} \longrightarrow W_{n+1} \longrightarrow 0.$$
Since $\widetilde{W}_{n+1} = \Stab{W_{n-1}}{W_n}$, the desired conclusion follows if $n = N+1$.  Assume
now that $n \geq N+2$.  To prove the desired conclusion, we must show that $Q_{n+1}' = 0$.  
Let $\partial_{n-2} : \Induce{V_{n-2}}{3} \rightarrow \Induce{V_{n-1}}{2}$ and
$\hat{\partial}_{n-2} : \Induce{V_{n-2}/W_{n-2}}{3} \rightarrow \Induce{V_{n-1}/W_{n-1}}{2}$ be the
$(n+1)$-boundary maps associated to $\phi_{n-2}$ and $\hat{\phi}_{n-2}$, respectively.  
Lemma \ref{lemma:chaincomplex} implies that $\partial_{n-1} \circ \partial_{n-2} = 0$.  
Observe that we have the following commutative diagram.
\begin{equation}
\label{eqn:bigdiagram}
\xymatrix{
\Induce{V_{n-2}}{3} \ar[r]^-{\partial_{n-2}} \ar@{->>}[dd] & \Induce{V_{n-1}}{2}         \ar[rr]^-{\partial_{n-1}} \ar@{->>}[dd]                              &                              & \Induce{V_n}{1} \ar@{->>}[dd] \ar@{->>}[ld]\\
                                                           &                                                                                                  & \widetilde{V}_{n+1}  \ar@{->>}[dr] &                                      \\
\Induce{V_{n-2}/W_{n-2}}{3} \ar[r]^-{\hat{\partial}_{n-2}} & \Induce{V_{n-1}/W_{n-1}}{2} \ar[ur]^-{\overline{\partial}_{n-1}} \ar[rr]^-{\hat{\partial}_{n-1}} &                               & \Induce{V_n/W_n}{1}                       }
\end{equation}
Since $\partial_{n-1} \circ \partial_{n-2} = 0$, we can chase this diagram to see
that $\overline{\partial}_{n-1} \circ \hat{\partial}_{n-2} = 0$, i.e.\ that $\Image(\hat{\partial}_{n-2}) \subset \Ker(\overline{\partial}_{n-1})$.  Let
$\overline{\partial}'_{n-1} : \Induce{V_{n-1}/W_{n-1}}{2} / \Image(\hat{\partial}_{n-2}) \rightarrow \widetilde{V}_{n+1}$
be the induced map.  

There are now two cases.  If $n \geq N+3$, then by induction the maps
$\hat{\phi}_{n-2}$ and $\hat{\phi}_{n-1}$ are stabilization maps, so Proposition \ref{proposition:spechtexactness} implies
that the bottom row of \eqref{eqn:bigdiagram} is exact.  This implies that the composition
$$\Induce{V_{n-1}/W_{n-1}}{2} / \Image(\hat{\partial}_{n-2}) \stackrel{\overline{\partial}_{n-2}'}{\longrightarrow} \widetilde{V}_{n+1} \longrightarrow \Induce{V_n/W_n}{1}$$
is injective.  Since $\widetilde{W}_{n+1} = \Ker(\widetilde{V}_{n+1} \rightarrow \Induce{V_n/W_n}{1})$, we conclude
that 
\[Q_{n+1}' = \Image(\overline{\partial}_{n-2}) \cap \widetilde{W}_{n+1} = 0.\]
If $n = N+2$, then $\hat{\phi}_{n-1}$ is a stabilization map but $\hat{\phi}_{n-2}$ need not be.  However, Lemma \ref{lemma:weakstabilization} says
that we can factor $\hat{\phi}_{n-1}$ as a composition of a surjection with a stabilization map, and just like in Step 2 we can
use this to run the above argument and get that $Q_{n+1}' = 0$, as desired.
\end{proof}

\begin{table}
\begin{center}
$$\xymatrixrowsep{1pc}\xymatrix{
V_2 \ar@{->}[r] & V_3 \ar@3{->}[r]                        & V_4 \ar@3{->}[r]                           & V_5 \ar@3{->}[r]                           & V_6 \ar@3{->}[r]                           & V_7 \ar@3{->}[r]                           & \cdots \\
                & W^1_3 \ar@{->}[r] \ar@{^{(}->}[u]^{k_1} & W^1_4 \ar@2{->}[r] \ar@{^{(}->}[u]^{k_1+1} & W^1_5 \ar@3{->}[r] \ar@{^{(}->}[u]^{k_1+2} & W^1_6 \ar@3{->}[r] \ar@{^{(}->}[u]^{k_1+3} & W^1_7 \ar@3{->}[r] \ar@{^{(}->}[u]^{k_1+4} & \cdots \\
                &                                         & W^2_4 \ar@{->}[r] \ar@{^{(}->}[u]^{k_2}    & W^2_5 \ar@2{->}[r] \ar@{^{(}->}[u]^{k_2+1} & W^2_6 \ar@3{->}[r] \ar@{^{(}->}[u]^{k_2+2} & W^2_7 \ar@3{->}[r] \ar@{^{(}->}[u]^{k_2+3} & \cdots \\
                &                                         &                                            & W^3_5 \ar@{->}[r] \ar@{^{(}->}[u]^{k_3}    & W^3_6 \ar@2{->}[r] \ar@{^{(}->}[u]^{k_3+1} & W^3_7 \ar@3{->}[r] \ar@{^{(}->}[u]^{k_3+2} & \cdots \\
                &                                         &                                            &                                            & W^4_6 \ar@{->}[r] \ar@{^{(}->}[u]^{k_4}    & W^4_7 \ar@2{->}[r] \ar@{^{(}->}[u]^{k_4+1} & \cdots \\}$$
\end{center}
\caption{The $V_i$ and $W^j_i$ for $N=3$.  Triple horizontal arrows are central stabilizations, double horizontal arrows
are quotients of central stabilizations, and numbers on the vertical arrows are the
cowidths of the constant cowidth subrepresentations}
\label{table:centraltospecht}
\end{table}

\begin{proof}[{Proof of Theorem \ref{maintheorem:centraltospecht}}]
Let us first recall the setup.  We have a coherent sequence
\begin{equation}
\label{eqn:seqinquestion}
V_1 \stackrel{\phi_1}{\longrightarrow} V_2 \stackrel{\phi_2}{\longrightarrow} V_3 \stackrel{\phi_3}{\longrightarrow} V_4 \stackrel{\phi_4}{\longrightarrow} \cdots
\end{equation}
of representations of the symmetric group over $\Field$ which is centrally stable
starting at $N$.  Also, we have either $\Char(\Field)=0$ or $\Char(\Field) \geq 2N+2$.  Our
goal is to prove that \eqref{eqn:seqinquestion} is Specht stable starting at $2N+1$.

By assumption, the sequence
$$V_{N-1} \stackrel{\phi_{N-1}}{\longrightarrow} V_N \stackrel{\phi_N}{\longrightarrow} V_{N+1} \stackrel{\phi_{N+1}}{\longrightarrow} V_{N+2} \stackrel{\phi_{N+2}}{\longrightarrow} \cdots$$
is the central stabilization sequence associated to $\phi_{N-1}$.  Let $k_1$ be the maximal width
of $V_N$, which is well-defined by our assumptions on $\Char(\Field)$.  Clearly $k_1 \leq N$.  For
$n \geq N$, let $W_n^{1} < V_n$ be the constant cowidth $k_1+(n-N) \leq n$ subrepresentation  
given by Lemma \ref{lemma:firstrow}.  Let $k_2$ be the maximal width of $W_{N+1}^1$, which again
is well-defined.  Since $W_{N+1}^1$ has width at most $k_1$ by assumption, we see that $k_2 \leq N$.
The sequence
$$W_N^1 \longrightarrow W_{N+1}^1 \longrightarrow W_{N+2}^1 \longrightarrow W_{N+3}^1 \longrightarrow \cdots$$
is a quotiented central stabilization sequence, so we can apply Lemma \ref{lemma:firstrow} again
and obtain constant cowidth $k_2+(n-N-1)$ subrepresentations $W_n^2 < W_n^1$ for $n \geq N+1$.

By our assumptions on $\Char(\Field)$, this process can be repeated several times to obtain
$W_n^i$ for $n \geq N-1+i$ and $1 \leq i \leq N+1$.  Here $W_n^{i+1}$ is a constant cowidth
$k_{i+1} + (n-N-i)$ subrepresentation of $W_n^i$, where $k_{i+1} \leq N$.  To help
keep all of this straight, see Table \ref{table:centraltospecht}.  Now, by assumption
$W^{N+1}_{2N}$ (resp.\ $W^{N+1}_{2N+1}$) has width at most $k_{N+1}-1 \leq N-1$ (resp.\ $k_{N+1} \leq N$).  
Since $W^{N+1}_{2N}$ (resp.\ $W^{N+1}_{2N+1}$) is a subrepresentation of $V_{2N}$ (resp.\ $V_{2N+1}$),
Lemma \ref{lemma:widthlowerbound} implies that $W^{N+1}_{2N}=0$ and $W^{N+1}_{2N+1}=0$.  But
this implies that $W^{N+1}_{i} = 0$ for all $i \geq 2N$.  It follows that for $n \geq 2N$ we have
a filtration
$$V_n \supset W^1_n \supset W^2_n \supset \cdots \supset W^{2N+1}_n = 0.$$
This might not quite be a top-indexed Specht filtration (for example, if $k_1 < N$), but we can
obtain one by adding repeated terms as necessary.  Our assumptions then imply that with
respect to these filtrations the maps $V_n \rightarrow V_{n+1}$ are stabilization maps
for $n \geq 2N+1$, and we are done.
\end{proof}

\section{Central stability implies polynomial dimensions}
\label{section:polynomial}

We now prove Theorem \ref{maintheorem:polynomial}.  By Theorem \ref{maintheorem:centraltospecht}, it
is enough to prove that if $\mu \vdash n$, then there is a polynomial $\phi(k)$ such
that $\phi(k) = \Dim S^{\stab^k(\mu)}(\Field)$ for $k \geq 0$.  This follows easily from
the results in \cite{ChurchEllenbergFarb}, but we give a short direct proof.
If $\nu$ is a partition, then let $\text{ST}(\nu)$ be the set of standard tableau of shape $\nu$, so 
$|\text{ST}(\nu)| = \Dim S^{\nu}(\Field)$.  If $t$ is a tableau, then denote by 
$\text{UR}(t)$ the entry in the upper right hand corner of $t$.  Finally, define
\begin{align*}
X_{k} = \{\text{$(t,\sigma)$ $|$ }&\text{$t \in \text{ST}(\mu)$ and $\sigma : \{1,\ldots,n\} \rightarrow \{1,\ldots,n+k\}$ is an}\\
&\text{order-preserving injection with $\sigma(i)=i$ for $1 \leq i \leq \text{UR}(\mu)$}\}.
\end{align*}
We will construct a bijection $\psi_k : \text{ST}(\stab^k(\mu)) \rightarrow X_k$.  Since there are
$\binom{n-m+k}{k}$ order-preserving injections $\sigma : \{1,\ldots,n\} \rightarrow \{1,\ldots,n+k\}$ such that
$\sigma(i)=i$ for $1 \leq i \leq m$, it will follow that
$$\Dim S^{\stab^k(\mu)}(\Field) = |\text{ST}(\stab^k(\mu))| = \sum_{t \in \text{ST}(\mu)} \binom{n-\text{UR}(t)+k}{k},$$
a polynomial in $k$.

For $s \in \text{ST}(\stab^k(\mu))$, define $\psi_k(s)=(t,\sigma)$, where $t$ and $\sigma$ are as follows.
Let $s_2$ be the result of deleting the last $k$ boxes from the first row of $s$.  Let $\sigma : \{1,\ldots,n\} \rightarrow \{1,\ldots,n+k\}$ be the unique
order-preserving injection whose image is the set of entries of $s_2$.  Finally, let $t$ be the
result of replacing each entry $i$ in $s_2$ with $\sigma^{-1}(i)$.  It is easy to see that
$(t,\sigma) \in X_k$.  To see that $\psi_k$ is a bijection, define a map 
$\phi_k : X_k \rightarrow \text{ST}(\stab^k(\mu))$ via $\phi_k(t,\sigma)=s$, where $s$ is obtained
by first replacing each entry $i$ in $t$ with $\sigma(i)$ and then appending the numbers
$\{1,\ldots,n+k\} \setminus \Image(\sigma)$ (in order) to the end of the first row of the result.  Clearly
$\phi_k$ is a $2$-sided inverse to $\psi_k$.

{\raggedright
Andrew Putman\\
Department of Mathematics\\
Rice University, MS 136 \\
6100 Main St.\\
Houston, TX 77005\\
E-mail: {\tt andyp@rice.edu}}

\end{document}